\documentclass[final]{siamart1116}

\usepackage[utf8x]{inputenc}

 \usepackage[T1]{fontenc}

\usepackage{amsfonts}
\usepackage{graphicx}
\usepackage{epstopdf}
\usepackage{algorithmic}

\usepackage{amsmath}            
\usepackage{amssymb}            
\usepackage{mathrsfs}           
\usepackage{mathtools}          

\usepackage{color}              
\usepackage{tikz}
\usepackage{pgfplots}
  \usetikzlibrary{positioning}
  \pgfplotsset{compat=newest}
\usepackage{xcolor}
	\definecolor{bluecolor}{RGB}{122,166,218}
	\definecolor{redcolor}{RGB}{213, 78, 83}
	\definecolor{orangecolor}{RGB}{231, 140, 69}
	\definecolor{greencolor}{RGB}{185, 202, 74}
	\definecolor{purplecolor}{RGB}{195, 151, 216}

\newcommand{\greencolor}{black!30!green}
\newcommand{\orangecolor}{black!30!orange}
\newcommand{\redcolor}{black!30!red}
\newcommand{\darkbluecolor}{black!50!blue}

\newcommand{\lightbluecolor}{white!50!blue}

\usepackage{algorithmic}
\Crefname{ALC@unique}{Line}{Lines}

\usepackage{cite}
\newsiamremark{example}{Example}
\newsiamremark{problem}{Problem}
\crefname{problem}{Problem}{Problems}

\newcommand{\bs}[1]{{\boldsymbol #1}}
\newcommand{\D}{{\mathcal D}}
\newcommand{\eps}{{\varepsilon}}
\DeclareMathOperator*{\argmin}{arg\,min}

\newcommand{\diag}{{\operatorname{diag}}}

\title{{Sparse solutions in optimal control of PDEs with uncertain
    parameters: the linear case}\thanks{Submitted to the editor \today\funding{Supported in part by the National Science Foundation
      under grants CBET-1507009 and DMS-1723211, and by the
      U.S.\ Department of Energy Office of Science, Advanced
      Scientific Computing Research (ASCR), Scientific Discovery
      through Advanced Computing (SciDAC) program.}}}

\author{Chen Li%
  \thanks{Courant Institute of Mathematical Sciences, New York
    University; \email{lichen@cims.nyu.edu}, \email{stadler@cims.nyu.edu}}%
  \and
  Georg Stadler%
  \footnotemark[2]
}

\begin{document}

\maketitle

\begin{abstract}
  We study sparse solutions of optimal control problems governed by
  PDEs with uncertain coefficients.  We propose two formulations,
  one where the solution is a deterministic control optimizing the
  mean objective, and a formulation aiming at stochastic controls that
  share the same sparsity structure.  In both formulations, regions
  where the controls do not vanish can be interpreted as optimal
  locations for placing control devices. In this paper, we focus on
  linear PDEs with linearly entering uncertain parameters.
  Under these assumptions, the deterministic formulation
  reduces to a problem with known structure, and thus we mainly focus on the
  stochastic control formulation. Here, shared sparsity is achieved by
  incorporating the $L^1$-norm of the mean of the pointwise squared
  controls in the objective. We reformulate the problem using a norm
  reweighting function that is defined over physical space only and
  thus helps to avoid approximation of the random
  space using samples or quadrature. We show that a fixed
  point algorithm applied to the norm reweighting formulation leads to
  a variant of the well-studied iterative reweighted least squares
  (IRLS) algorithm, and we propose a novel preconditioned
  Newton-conjugate gradient method to speed up the IRLS
  algorithm. We combine our algorithms with low-rank
  operator approximations, for which we provide estimates of the
  truncation error. We carefully examine the computational complexity
  of the resulting algorithms. The sparsity
  structure of the 
  optimal controls and the performance of the solution algorithms are
  studied numerically using control problems governed by the Laplace
  and Helmholtz equations. In
  these experiments the Newton variant clearly
  outperforms the IRLS method.
\end{abstract}

\begin{keywords}
Optimal control of PDEs, uncertainty, sparse controls, iterative
reweighting, $L^1$-minimization, Newton method
\end{keywords}

\begin{AMS}
  60H35, 
  35Q93, 
  35R60, 
  49M15, 
  49J52 
\end{AMS}


\section{Introduction}
\label{sec:intro}
Solving optimal control problems governed by partial differential
equations (PDEs) that contain uncertain parameters represents a
significant challenge. However, thanks to theoretical and algorithmic
advances, and to the ever increasing availability of computing resources,
significant progress has been made over the last decade
\cite{Kouri12, BorziVonWinckel11, KouriHeinkenschlossRidzalEtAl13,
  BennerOnwuntaStoll16, NegriRozzaManzoniEtAl13, AliUllmannHinze17,
  GarreisUlbrich17, TieslerKirbyXiuEtAl12, ChenQuarteroni14}.
  In this paper, we aim at optimal control
problems under uncertainty,
where the control objective
involves a sparsifying term and, as a consequence, distributed optimal
controls vanish on parts of the domain. The areas where controls are
nonzero are interpreted as locations where it is most efficient to
employ control devices \cite{Stadler09, HerzogStadlerWachsmuth12,
  ClasonKunisch12, 
  Casas17, BrunnerClasonFreibergerEtAl12, CaponigroFornasierPiccoli15,
CasasRyllTroltzsch13, CiaramellaBorzi16}.

Given a physical domain $\D\subset \mathbb R^n$, $n\in \{1,2,3\}$, we consider
a partial differential equation involving
uncertain parameters written as
\begin{equation}\label{eq:state}
c(y,u,m(\omega)) = 0.
\end{equation}
Here, $u$ and $y$ are the control and the state variables,
respectively, $m(\omega)$ is an uncertain parameter, and
$c(\cdot\,,\cdot\,,\cdot)$ denotes the PDE relating these variables.
We assume that the distribution law of $m$, denoted by $\mu$, is
supported on a Hilbert space $\mathscr H$, and consider $m$ as an
$\mathscr H$-valued random variable.  That is, for a probability space
$(\Omega, \mathcal F, \mu)$, $m : \Omega \to \mathscr H$, where
$\Omega$ is the set of events, $\mathcal F$ a $\sigma$-algebra of sets
in $\Omega$, and $\mu$ is a positive normalized measure.  For example,
$\mathscr H$ can be chosen as infinite-dimensional function
space over $\D$ or its boundary $\partial \D$.
%
We assume that for every $u\in L^2(\D)$ and $\omega\in \Omega$,
\eqref{eq:state} has a unique (weak) solution $y=y(u;\omega,\cdot) \in
V$, with an appropriate space $V\subset L^2(\D)$.\footnote{While the
  discussion is kept general in this introduction, in most of
  the remainder of this paper we focus on linear equations, where this
  assumption can easily be verified.}
%
For $\omega\in \Omega$,  we consider the
optimal control problem in reduced form,
\begin{equation}\label{eq:optcon}
  \min_{u\in U_{\!\text{ad}}} J(\omega, u) := \frac 12
  \int_\D (y(u;\omega,\cdot)
  - y_d)^2 \,d\bs x + \frac\alpha 2 \int_\D u^2\,d\bs x.
\end{equation}
Here, $U_{\!\text{ad}} = \{u\in L^2(\D): a\le u\le b \text{ a.e.}\}$,
where $a,b\in L^2(\D)$ with $a<b$ almost everywhere. Moreover,
$\alpha>0$ is a regularization/control cost parameter and $y_d\in L^2(\D)$ a given
desired state.
For each $\omega\in \Omega$, \eqref{eq:optcon} is a classical
control-constrained linear-quadratic optimal control
problem. As is well known, this problem has a unique solution that
depends on $\omega$.

We are interested in distributed optimal control
problems, where the controls are sparse, i.e., they vanish on parts of
the domain. We propose two practically relevant approaches to
sparse optimal control under uncertainty.
The first computes a deterministic sparse
control that is optimal for the expectation of the cost functional.
The second aims at stochastic controls that depend on the
uncertain parameter, but have shared sparsity structure.

\subsection{Deterministic sparse optimal control}
\label{sec:deterministic}
Robust deterministic controls are optimal in expectation
\cite{Kouri12, BorziVonWinckel11, KouriHeinkenschlossRidzalEtAl13}, or
optimal with respect to a risk measure
\cite{AlexanderianPetraStadlerEtAl17, KouriSurowiec16}.  Since in this
formulation the controls are deterministic, it is straightforward to
add a sparsity-enhancing term for the control to \eqref{eq:optcon}, resulting in
\begin{equation}\label{eq:optcon_d}
  \min_{u\in U_{\!\text{ad}}} \mathcal J_{{d}}(u) = \frac 12
  \int_\Omega \int_\D (y(u;\omega,\cdot) - y_d)^2 \,d\bs x\,d\mu
  + \frac\alpha 2 \int_\D u^2\,d\bs x
+ \beta\int_\D |u|\,d\bs x.
\end{equation}
Here, $\beta>0$ is the weight for the sparsity-enhancing $L^1$-term,
in which $|\cdot|$ denotes the absolute value. The deterministic
optimal controls found from this formulation vanish on parts of the
spatial domain $\D$, and the value of $\beta$ influences how sparse
the control are. The resulting control structure can be used to decide
on the placement of control devices. In the deterministic context,
extensions of this approach have been applied for instance to optimal
device placement in tissue imaging
\cite{BrunnerClasonFreibergerEtAl12}, sparse control of alignment
models \cite{CaponigroFornasierPiccoli15}, optimal control of
traveling wave fronts \cite{CasasRyllTroltzsch13} or shaping controls
for quantum systems \cite{CiaramellaBorzi16}.

\subsection{Stochastic optimal control with shared sparsity}
\label{sec:stochastic}
An alternative problems class and this paper's main focus is to find
stochastic\footnote{While stochasticity is often used in the context of
  time-dependent problems, here it simply means that the controls
  depend on the random variable, in contrast to the deterministic
  optimal control formulation discussed above.} controls $u\in
\bs U_{\!\text{ad}}:=\{u\in L^2_\mu(\Omega,L^2(\D)), a\le
u(\omega,\cdot)\le b \text{ a.e.}\}$, i.e., individual controls
$u=u(\omega)$ for each $\omega\in\Omega$ \cite{TieslerKirbyXiuEtAl12,
  BennerOnwuntaStoll16, NegriRozzaManzoniEtAl13, ChenQuarteroni14}.
These controls
minimize the expected objective value, i.e.,
\begin{equation}\label{eq:optcon_s1}
  \min_{u\in \bs U_{\!\text{ad}}} \int_\Omega J(\omega,u(\omega,\cdot)) \,d\mu,
\end{equation}
with $J(\cdot,\cdot)$ as defined in \eqref{eq:optcon}. Note
that \eqref{eq:optcon_s1} amounts to solving optimal control problems
of the form \eqref{eq:optcon} for each $\omega\in \Omega$, followed by
computing the expectation over the values of the objective obtained
with these optimal controls. One possibility for incorporating
sparsity in this formulation is to add sparsity-enhancing
regularization for each $\omega$. However, since we interpret regions
where optimal controls are non-zero as regions where we propose to
place control devices, it is more meaningful to require that the
stochastic controls {\em share} their sparsity structure. This can be
achieved by adding a sparsity-enforcing term to the
objective functional in \eqref{eq:optcon_s1}:
\begin{equation}\label{eq:optcon_s}
  \min_{u\in \bs U_{\!\text{ad}}} \mathcal J(u):=\int_\Omega J(\omega,u(\omega,\cdot)) \,d\mu
  +  \beta\int_\D
  \left(\int_\Omega|u(\omega,\cdot)|^2\,d\mu\right)^{\!\frac 12 }
  \,d\bs x.
\end{equation}
Here, $\beta>0$ and the outer integral in the sparsity-enforcing term
is over the pointwise marginal distribution of the squared
controls. Note that the sparsity term is well-defined and finite for
$u\in \bs U_{\!\text{ad}}$.  As will be shown, using the $L^1$-norm of
the pointwise expectation results in optimal controls
$u(\omega,\cdot)$ with shared sparsity.  While the optimal controls
are stochastic, i.e., they depend on $\omega$, the controller
locations resulting from \eqref{eq:optcon_s} are deterministic, i.e.,
they only depend on the probability space, but not the individual
event $\omega\in \Omega$. A practical interpretation of this approach
is that the optimal location of controllers is computed by solving
\eqref{eq:optcon_s} in an offline phase, while the optimal controls
$u(\omega)$ are computed in an online phase corresponding to the
particular realization of the random variable $\omega$.

Let us give two application examples for an optimal control formulation
of the form \eqref{eq:J}. First, we consider a problem from
earthquake engineering, where one wants to find locations for active
damping devices (controllers) that shall dampen vibrations
that originate from an unknown earthquake forcing. In this situation,
one can imagine that the optimal controls can be computed in real time
individually for each forcing, but the location of controllers needs
to be decided upfront and should be chosen in an optimal way for all
possible earthquakes. Another application could be to position heaters
in a building to obtain, for instance, a uniform temperature distribution
in the presence of uncertain heat sinks/sources caused, e.g., by
open windows, leaky walls or the presence of people.

\subsection{Related work}
Optimization under uncertainty governed by partial differential
equations has been an active field of research over the last decade.
Various formulations are proposed in the literature, e.g., robust
deterministic optimal control \cite{GarreisUlbrich17,
  AlexanderianPetraStadlerEtAl17, Kouri12,
  KouriSurowiec16} and stochastic control
\cite{BennerOnwuntaStoll16, ChenQuarteroni14,
  TieslerKirbyXiuEtAl12}.  The main focus of this paper is a
stochastic control problem with linear governing equations, but
with the additional requirement that the optimal controls are jointly
sparse.

The interest in sparse optimal control is (1) due to its application
for control device placement and its ability to discover controls with simple
structure \cite{Stadler09, ClasonKunisch12,
  BrunnerClasonFreibergerEtAl12, CaponigroFornasierPiccoli15,
  CasasRyllTroltzsch13, Pieper15}, and (2) due to the interesting non-reflexive
Banach space structure that arises if no Hilbert-space norm term is
added to the objective
\cite{ClasonKunisch12,Casas17,CasasClasonKunisch12}. The stochastic
sparse control problem we study has similarities with
the notion of directional sparsity proposed for optimal control of parabolic
problems \cite{HerzogStadlerWachsmuth12,KunischPieperVexler14}, in which one has to decide
on the sparsity for an entire time stripe.  However, differently from
directional sparsity problems, our stochastic control formulation
requires to decide on the sparsity based on a potentially
high-dimensional integration over the probability space rather than an
integration over the one-dimensional time direction.


The solution methods we propose are related to iteratively reweighted
least squares (IRLS) algorithms, which are used, e.g., in compressive
sensing, image processing and matrix recovery
\cite{CandesWakBoyd2008, DaubechiesDeVoreFornasierEtAl10,
  KarthikMaryam2012}. While IRLS methods have mostly been used in
finite dimensions and in the context of underdetermined problems, they
have recently also been studied for
infinite-dimensional $L^p$ and $\ell^p$ ($p\le 1$) optimization
\cite{ItoKunisch14, ItoKunisch13}. Instead of solving an $L^p$-problem
directly, iterative reweighting algorithms alternate between
solving a simpler (e.g., a quadratic minimization) problem,
and updating a weighting function that enters into this simpler
problem. In this paper we focus on a convex problem, but a significant
challenge is that the optimization variable is defined over physical
{\em and} infinite/high-dimensional random space. Hence, we employ ideas
from iterative reweighting to avoid working with the
high-dimensional optimization variable. The proposed algorithms only
iterate over the reweighting function, which does not depend on the
random variable. Additionally to the first-order IRLS method
\cite{DaubechiesDeVoreFornasierEtAl10}, we propose
a Newton-type algorithm based on
the reweighting formulation and show that it outperforms the classical
IRLS iteration.

To accelerate norm reweighting methods, they have been combined with
an active set method \cite{ItoKunisch13}. This approach also applies
to nonconvex problems, but it requires the computation of norms of the
optimization variable and the corresponding dual variable, similarly
as for Newton-type methods for directional sparsity problems
\cite{HerzogStadlerWachsmuth12}. In the context of control under
high-dimensional uncertainty, these norms are integrals over random
space, making this integration a computational challenge.


Another attempt to accelerate the IRLS
algorithm is to use the conjugate gradient method for
solution of the auxiliary least squares problems that occurs in each iteration
\cite{FornasierPeterRauhutEtAl16}. In \cref{sec:newton}, we propose a
related idea, but instead
of iterative linear solves, we exploit the optimal control
problem structure. Namely, we combine a low-rank operator
approximation and the Sherman-Morrison-Woodbury identity to design
fast (i.e., optimal complexity) iteration algorithms.


\subsection{Contributions and limitations}
The main contributions of this paper are as follows. (1) We propose a
formulation for stochastic optimal controls with shared sparsity in
the presence of uncertain parameters in the governing PDE. We believe
that this formulation is relevant in applications and that it is
interesting from the optimization-under-uncertainty perspective since
the joint sparsity requirement couples the controls for different
uncertain parameters, and from the sparse-control perspective due to
the infinite-dimensional non-differentiable optimization
structure. (2) We propose a Newton-type variant of the
iteratively reweighted least squares (IRLS) minimization algorithm for
the solution of this optimization problem, which is significantly
faster than the classical IRLS method. (3) We present low-rank operator
approximations which allow fast iterations of the IRLS algorithm and
its Newton variant. We also provide bounds for the truncation error
due to these approximations.


Next, we summarize limitations of this work. (1) We restrict ourselves
to linear governing equations with uncertain parameters that
enter linearly in these equations, and our algorithms for stochastic
control with shared sparsity do not allow for control
constraints.  More general problems are a
significant challenge even for optimal control without sparsity
requirements on the controls. (2) The proposed norm reweighting
problem reformulation requires regularization of the
non-differentiable term in the objective and convergence to a truly
sparse solution requires that the regularization parameter is driven
to zero. (3) Our arguments and in particular the low-rank operator
approximation require that the parameter $\alpha$ in
\eqref{eq:optcon_s} is positive. Using $\alpha=0$ does not
allow a Hilbert space formulation and would be more challenging from
the theoretical as well as the computational perspective.


\subsection{Notation}
We consider a probability space
$(\Omega,\mathcal F,\mu)$, i.e., $\Omega$ is the set of events,
$\mathcal F$ a $\sigma$-algebra of sets in $\Omega$, and $\mu$ a
positive normalized measure.  For $1\le p<\infty$ and a Banach space
$(X,\|\cdot\|_X)$, the Bochner space $L_\mu^p(\Omega;X)$ is the
space of Bochner integrable functions $u:\Omega\to X$,
for which $\left(\int_\Omega\|u(\omega)\|^p_X\,d\mu\right)^{1/p}:=\|u\|_{L^p(\Omega,X)}$
is finite. This value has the properties of a norm.
For a domain $\D\subset \mathbb R^d$, $d\in \{1,2,3\}$, we will in
particular use the spaces 
\begin{equation}
\bs V := L_\mu^2(\Omega;L^2(\D)) \text{ and } \bs Y := L_\mu^2(\Omega;Y),
\end{equation}
where $Y\subset H^1(\D)$ is a subspace that can includes Dirichlet
boundary conditions on part of $\partial \D$.  Both, $\bs V$ and $\bs Y$ are
Hilbert spaces with inner products derived from the inner products in
$X$, e.g., the inner product for $\bs V$ is $\langle u,v
\rangle_{\bs V}:=\int_\Omega(u(\omega,\cdot),v(\omega,\cdot))_{L^2(\D)}\,d\mu$
and the induced norm for $u\in \bs V$ is $\|u\|_{\bs V} = \langle u,u
\rangle_{\bs V}^{1/2}$.
For $u\in \bs V$, we will commonly use the notation
\begin{equation}\label{eq:normOmega}
\|u\|_\Omega = \|u\|_\Omega(\bs x) =
\left(\int_\Omega u(\omega,\bs x)^2\,d\mu\right)^{\!1/2},
\end{equation}
which is well-defined due to the isomorphism between Bochner spaces
and spaces defined over the product space $\Omega\times\D$
\cite{HytonenNeervenVeraarWeis16}. In the remainder of this paper, we
use bold letters for function spaces defined over
$\Omega\times\D$, such as $\bs V$ and $\bs Y$.


\section{PDE with linearly-entering uncertain parameters}
The main focus of this paper is on problems where the uncertain
parameters enter linearly in \eqref{eq:state}. We allow for
infinite-dimensional uncertain parameters that follow a Gaussian
distribution $\mu=\mathcal N(m_0, \mathcal C_0)$ over a Hilbert space
$\mathscr H$, where $m_0\in \mathscr H$, and $\mathcal C_0$ is a
self-adjoint, positive definite trace class operator over $\mathscr
H$. We consider $m$ as an $\mathscr H$-valued random variable and with
a slight abuse of notation, we denote realization of this random
variable by the same symbol $m$.  We consider a linear differential
equation of the form
\begin{equation}\label{eq:redstate}
A y = u + f + B m,
\end{equation}
where $A:Y\subset H^1(\D)\mapsto H^{-1}(\D)$ is 
invertible, $B:\mathscr H\mapsto H^{-1}(\D)$ and $f,u\in H^{-1}(\D)$.  The following two
examples fit into this framework.
\begin{example}[Poisson problem with uncertain Robin boundary data] \label{ex1:poisson} As example for an
  equation of the form \eqref{eq:redstate}, we use $\mathscr
  H=L^2(\partial\D_2)$ and consider a problem with inhomogeneous Robin
  boundary condition with uncertain data on $\partial\D_2$:
\begin{subequations}\label{eq:state_linear}
\begin{alignat}{2}
  -\nabla \cdot \left(a(\bs x)\nabla y(\omega,\bs x)\right) &=
  f(\bs x) + u(\bs x) \qquad &&\text{in } \D, \label{eq:state_linear1}\\
  y(\omega,\bs x)  &= 0  &&\text{in }\partial\D_1, \label{eq:state_linear2}\\
  k y(\omega,\bs x) + \left(a(\bs x) \nabla y(\omega,\bs x)\right)\cdot{\bs n} &= m(\omega) &&\text{in
  }\partial\D_2. \label{eq:state_linear3}
  \end{alignat}
\end{subequations}
Here, $a(\bs x)\ge a_0>0$ and $k\ge 0$. For
$k=0$, the Robin condition on $\partial\D_2$ reduces to a Neumann
boundary condition.
\end{example}
\begin{example}[Poisson problem with uncertain right hand side]
  \label{ex1:poisson_rhs} This example coincides with 
  \cref{ex1:poisson} with the exception that the uncertain parameter
  field enters on the right hand side of \eqref{eq:state_linear1}. We
  use $\mathscr H=L^2(\tilde\D)$ with $\tilde\D\subset \D$ an open
  subdomain, and $B:L^2(\tilde\D)\to L^2(\D)$ is the
  extension-by-zero operator. The only difference to the problem above is
  that the term $Bm(\omega)$ is added to the right hand side in
  \eqref{eq:state_linear1}, whereas the right hand side of
  \eqref{eq:state_linear3} is zero.
\end{example}
\begin{example}[Helmholtz problem with uncertain Neumann boundary data]\label{ex1:helmholtz}
  Another example that fits into our framework is a Helmholtz
  problem with uncertain Neumann boundary forcing on $\partial\D_2$:
\begin{subequations}
  \begin{alignat}{2}
    -\Delta y(\omega,\bs x) - \kappa^2y(\omega,\bs x) &= u(\bs x) \qquad &&\text{in } \D, \label{eq:helmhotz1}\\
    y(\omega,\bs x)  &= 0  &&\text{in }\partial\D_1, \label{eq:helmholtz2}\\
    \nabla y(\omega,\bs x)\cdot{\bs n} &= m(\omega) &&\text{in
    }\partial\D_2. \label{eq:helmholtz3}
  \end{alignat}
\end{subequations}
Here, $\Delta$ is the Laplace operator and $\kappa>0$ is the wave number.
\end{example}

Most formulations and solution methods in this paper can
straightforwardly be
extended to generalizations of \eqref{eq:redstate}, e.g.,
to problems
where $u$ is a vector function as in linear elasticity. However,
for simplicity of the presentation, we restrict ourselves to
\eqref{eq:redstate}. Next, we specialize the deterministic and
stochastic optimal control formulations from \cref{sec:deterministic}
and \cref{sec:stochastic} for this linear case.

\subsection{Deterministic sparse optimal control}
Using the linear equation \eqref{eq:redstate}, the deterministic
sparse optimal control problem becomes
\begin{equation}\label{eq:optcon_linear_d}
  \min_{u\in U_{\!\text{ad}}} \frac 12 \int_\Omega \int_\D (A^{-1}u +
  A^{-1}Bm - \hat y_d)^2 \,d\bs{x}\,d\mu
  + \frac\alpha 2 \int_\D u^2\,d\bs{x}
  + \beta\int_\D |u|\,d\bs{x},
\end{equation}
where $\hat y_d = y_d - A^{-1}f$.  Since the integration of a
quadratic form over a Gaussian random variable can be done analytically
\cite[Remark 1.2.9]{DaPratoZabczyk02}, one obtains
\begin{multline}\label{eq:d1}
 \int_\Omega \int_\D (A^{-1}u + A^{-1}Bm - \hat y_d)^2
 \,d\bs{x}\,d\mu \\=
 \int_\D (A^{-1}u + A^{-1}Bm_0 - \hat y_d)^2 \,d\bs{x} +
  \text{Tr}\left(\mathcal C_0^{1/2} B^\star A^{-\star}  A^{-1}B \mathcal C_0^{1/2} \right),
\end{multline}
where $B^\star$ denotes the adjoint of $B$, $ A^{-\star}$ the
adjoint of $A^{-1}$, and $\text{Tr}(\cdot)$ denotes the trace of
an operator.  Note that since $\mathcal C_0$ is trace-class, and $B^\star
A^{-\star} A^{-1}B$ is bounded on $L^2(\Omega)$, the trace in
\eqref{eq:d1} is finite. Since the trace term does not depend on the
control, it can be neglected in the computation of the minimizer.
Thus, the optimal control derived from \eqref{eq:optcon_linear_d} is
equivalently characterized by the minimization problem
\begin{equation}\label{eq:optcon_linear_d1}
  \min_{u\in U_{\!\text{ad}}} \frac 12  \int_\D
  (A^{-1}u  - \tilde y_d)^2 \,d\bs{x} + \frac\alpha 2 \int_\D u^2\,d\bs{x}
  + \beta\int_\D |u|\,d\bs{x},
\end{equation}
where $\tilde y_d = y_d - A^{-1}f + A^{-1}Bm_0$. This is a
deterministic elliptic control problem with $L^1$-control cost, where
the desired state $\tilde y_d$ depends on the mean of the distribution
of $m$. Problems of this form and algorithms for their solution have
been studied for instance in
\cite{Stadler09}, and have been generalized in various
directions \cite{Casas17}. In particular, it is known that
\eqref{eq:optcon_linear_d1} admits a unique optimal control, which is
sparse in the sense that it vanishes on parts of the domain
$\D$. Moreover, this solution can be computed efficiently using a
semismooth Newton algorithm in function space.

\subsection{Stochastic optimal control with shared sparsity}
Next, we consider the stochastic optimal control formulation
\eqref{eq:optcon_s} for the linear governing equation
\eqref{eq:redstate}.
\begin{equation}\label{eq:optcon_linear_s}
  \min_{u\in \bs U_{\!\text{ad}}}
  \frac 12
  \int_\Omega \int_\D ( (A^{-1}u + A^{-1}Bm - \hat y_d)^2 + \alpha u^2)\,d\bs{x}\,d\mu
  +  \beta\int_\D\!
  \left(\int_\Omega u(\omega,\cdot)^2\,d\mu\right)^{\!\! 1/2} \!\! d\bs{x},
\end{equation}
where, as above $\hat y_d = y_d - A^{-1}f$. The added sparsity term is
an infinite-dimensional version of an $\ell^{2, 1}$-norm, used in
finite dimensions to achieve group sparsity structure or matrix
sparsity \cite{NieHuangXiaoChris10, JunShuiJie12}. In general, for
$r,p>0$, the $\ell^{r,p}$-norm
of a two-index array $a_{i,j}$, $1\le i\le n$, $1\le j\le m$ is defined as
\begin{equation*}
	\|a\|_{r, p} \coloneqq \Big[ \sum_{i=1}^n \Big( \sum_{j=1}^m |a_{i,j}|^r \Big)^{\frac p r} \Big]^{\frac 1 p}.
\end{equation*}
In \eqref{eq:optcon_linear_s}, we aim at obtaining shared sparsity
structure amongst
controls for different uncertain parameters. If the integrand vanishes
at a point $\bs x\in \D$, then the controls for almost all random
parameters must vanish at this $\bs x$, resulting in shared sparsity.
We also notice that the sparsifying term
couples the problems for different random variables $\omega$ and thus
one cannot integrate over the random space analytically. The
properties and solution algorithms for \eqref{eq:optcon_linear_s} are
the main focus of this paper. Note that this problem is related to the
directional sparsity formulation for time-dependent optimal control
proposed in \cite{HerzogStadlerWachsmuth12}, with the stochastic space
taking the role of the time direction. One of the main differences
between the directional sparsity and stochastic control with shared
sparsity is the high dimension of the probability space compared to
the one-dimensional time variable. Due to this difference, the
generalized Newton algorithms used to solve directional sparsity
control problems cannot be applied for the solution of
\eqref{eq:optcon_linear_s}. In the next section, we characterize
solutions to the shared sparsity control problem and introduce a
regularized variant of \eqref{eq:optcon_linear_s}.

\section{Properties of shared sparsity stochastic control problem}
\label{sec:properties}
We first introduce the notation $Q:U_{\!\text{ad}}\times \mathscr H\to
\mathbb R$,
\begin{equation*}
Q(u,m):=\frac 12 \int_\Omega \int_\D ( (A^{-1}u + A^{-1}Bm - \hat
y_d)^2 + \alpha u^2)\,d\bs x\,d\mu.
\end{equation*}
First, we summarize necessary and sufficient
optimality conditions for solutions to \eqref{eq:optcon_linear_s}
using the notation introduced in \eqref{eq:normOmega}.
\begin{theorem}\label{thm:optsystem}
  The optimal control problem \eqref{eq:optcon_linear_s} has a unique
  solution $\bar u\in \bs U_{\!\text{ad}}$, characterized by the
  existence of corresponding state $\bar y\in \bs{Y}$, adjoint state $\bar p\in
  \bs{Y}$ and multiplier $\bar \lambda\in \bs V$ such that
\begin{subequations}\label{eq:opts}
\begin{alignat}{2}
  A \bar y - \bar u - f - B m &= 0, \label{eq:opts1}\\
  A^\star \bar p - y_d + \bar y &= 0, \label{eq:opts2}\\
  -\bar p + \alpha \bar u + \beta \bar \lambda + \bar \mu &= 0, \label{eq:opts3}
\end{alignat}
\vspace{-5ex}
\begin{equation}\label{eq:optlambda1}
  \left.
  \begin{array}{lr}
  \bar\lambda(\omega,\bs x) = \displaystyle\frac{\bar u(\omega,\bs x)}{\|\bar u\|_\Omega(\bs{x})} &\text{
    for $\bs{x}\in \D$ with } \|\bar u\|_\Omega(\bs{x})  \not=0 \\[2ex]
  \|\bar \lambda\|_\Omega(\bs{x}) \le 1 &\text{ for $\bs{x}\in \D$ with }
  \|\bar u\|_\Omega(\bs{x}) = 0 
\end{array} \right\} \text{for a.a.\ $\bs x\in \D$},
\end{equation}
\vspace{-3ex}
\begin{alignat}{2}\label{eq:optmu}
  \bar \mu \le 0 \:\text{ if }\: \bar u = a, \quad &\bar \mu \ge 0 \:\text{
    if }\: \bar u = b, \:\: \text{and} \:\: \bar \mu =0 \:\text{ if }\:
  a\le \bar u \le b \:\: \text{a.e.\ in $\D\times \Omega$}.
\end{alignat}
\end{subequations}
\end{theorem}
\begin{proof}
We denote the sparsifying term in \eqref{eq:optcon_linear_s} by
$\varphi(u)$, i.e., $\varphi(u) := \int_\D\|u\|_\Omega\,d\bs{x}$.  It
follows from convex analysis \cite{EkelandTemam99} that the
variational inequality
\begin{equation}\label{eq:vari_ineq}
  \langle Q_{u}(\bar u, m), u - \bar u \rangle_\bs{V} + \beta (\varphi(u) -
  \varphi(\bar u)) \geq 0 \quad  \text{for all} \ u \in \bs
  U_{\!\text{ad}},
\end{equation}
is necessary and sufficient for $\bar u$ to be a solution of
\eqref{eq:optcon_linear_s}, where $Q_{u}$ denotes variation of $Q$
with respect to $u$. This is equivalent to
\begin{equation}\label{eq:vari_ineq_2}
  \langle A^{-\star} (A^{-1}\bar u + A^{-1} B m - \hat y_d) + \alpha
  \bar u + \beta \bar \lambda, u - \bar u \rangle_\bs{V} \geq 0 \quad  \text{for all} \
  u \in \bs U_{\!\text{ad}},
\end{equation}
where $\bar\lambda \in \bs{V}$ is an element in the subdifferential
$\partial \varphi(\bar u)$ of $\varphi$ at $\bar u$.  By introducing
the adjoint variable $\bar p$ and a Lagrange multiplier $\bar\mu
\in \bs{V}$ associated with the bound constraints in $\bs
U_{\!\text{ad}}$, \eqref{eq:vari_ineq_2} results in
\eqref{eq:opts1}--\eqref{eq:opts3} and \eqref{eq:optmu}.

It remains to show that $\bar\lambda \in \partial \varphi(\bar u)$ is
equivalent with \eqref{eq:optlambda1}.
Considering $\varphi:\bs{V}\to\mathbb R$, the subdifferential is
defined as
\begin{equation}\label{eq:subdifferential}
\partial \varphi(\bar u) = \{ \lambda \in \bs{V} \mid \langle \lambda, v
- \bar u\rangle_\bs{V} \leq \int_{\D} \left( \| v \|_{\Omega}
  - \| \bar u \|_{\Omega} \right) \,d\bs{x} \ \text{ for any } v \in \bs{V}\}.
\end{equation}
To show equivalence, let us first assume that $\bar\lambda\in
\partial\varphi(\bar u)$. 
Choosing $v := \bar u + \bar\lambda \delta$, where
$\delta \in L^\infty(\D)$, i.e., it only depends on $\bs{x}$, we
obtain:
$\int_{\D} \delta \|\bar\lambda\|^2_{\Omega} \,d\bs{x} = \langle\bar\lambda, \delta\bar\lambda\rangle_\bs{V} \leq \int_{\D} \delta \|\bar\lambda\|_{\Omega} \,d\bs{x}$,
which implies that $\|\bar\lambda\|_{\Omega} \leq 1$. Setting $v = 0$ in
\eqref{eq:subdifferential} shows that
\begin{equation*}
  \int_{\D} \|\bar u\|_\Omega \,d\bs{x} \leq \langle \lambda, \bar u
  \rangle_\bs{V} \leq \int_{\D} \|\bar \lambda\|_{\Omega} \|\bar u\|_{\Omega} \,d\bs{x}
  \leq \int_{\D} \|\bar u\|_{\Omega} \,d\bs{x},
\end{equation*}
where we have used H\"older's inequality for the second estimate.
For $\bs{x}\in \D$ 
with $\|\bar u\|_\Omega(\bs{x}) \not=0$, thus necessarily
$\bar\lambda(\omega,\bs{x}) = {\bar u(\omega,\bs{x})}/{\|\bar
  u\|_\Omega(\bs{x})}$. Thus, we have shown that $\bar\lambda\in
\partial\varphi(\bar u)$ implies that $(\bar u,\bar\lambda)$ satisfies
\eqref{eq:optlambda1}. Conversely, we assume that
$\bar u$ and $\bar\lambda$ satisfy \eqref{eq:optlambda1} and we split $\D$ into $\D_1 = \{ \bs{x} \in \D
\mid \|\bar u\|_{\Omega}(\bs{x}) = 0 \}$ and $\D_2 = \D \setminus
\D_1$. For any $v \in \bs{V}$, we then have on $\D_1$,
\begin{equation}\label{eq:D_1}
  \int_{\D_1} \int_{\Omega} \bar \lambda v \,d\mu d\bs{x} \leq \int_{\D_1}
  \|\bar\lambda\|_{\Omega} \|v\|_{\Omega} \,d\bs{x} \leq \int_{\D_1}
  \|v\|_{\Omega} \, d\bs{x} = \int_{\D_1}
  \|v\|_{\Omega} - \|\bar u\|_{\Omega} \, d\bs{x},
\end{equation}
and on $\D_2$,
\begin{equation}\label{eq:D_2}
  \int_{\D_2} \int_{\Omega} \bar \lambda (v - \bar u) \,d\mu d\bs{x} = \int_{\D_2} \int_{\Omega} \frac{\bar u}{\|\bar u\|_{\Omega}} (v - \bar u) \,d\mu d\bs{x} \leq \int_{\D_2} \|v\|_{\Omega} - \|\bar
  u\|_{\Omega} \, d\bs{x},
\end{equation}
where we have used the Cauchy-Schwartz inequality. Combining
\eqref{eq:D_1} and \eqref{eq:D_2}, we find that
$\bar\lambda\in\partial\varphi(\bar u)$, which ends the proof.
\end{proof}
We now define the following family of regularized control objectives
for $\eps\ge 0$.
\begin{equation}\label{eq:J}
  \mathcal J(u,\eps):= Q(u,m) + \beta\int_\D
  \left(\|u\|_\Omega^2 + \eps^2\right)^{1/2}\,d \bs{x}.
\end{equation}
In particular, the objective in
\eqref{eq:optcon_linear_s} is ${\mathcal J}(u,0)$.
%
A result similar to \cref{thm:optsystem} also holds for the
regularized problem \eqref{eq:J}, where the objective function is now
differentiable.
\begin{corollary}
  The regularized problem \eqref{eq:J} with $\eps>0$ has a unique
  solution $u_\eps\in \bs U_{\!\text{ad}}$, characterized by the
  existence of corresponding state $y_\eps\in \bs{Y}$, adjoint $p_\eps\in
  \bs{Y}$ and multiplier $\lambda_\eps\in \bs{V}$ such that
  \begin{subequations}
    \begin{alignat}{2}
      A y_\eps - u_\eps - f - B m &= 0, \label{eq:opts1reg}\\
      A^\star  p_\eps - y_d + y_\eps &= 0, \label{eq:opts2reg}\\
      -p_\eps + \alpha u_\eps + \beta \frac{u_\eps}{\sqrt{\|u_\eps\|_\Omega^2+\eps^2}} + \mu_\eps &=
      0 \:\:\text{ for a.a.\ $\bs x\in \D$}, \label{eq:opts3reg} \\
      \mu_\eps \le 0 \text{ if } u_\eps = a, \:\: \mu_\eps \ge 0 \text{
        if } u_\eps = b, &\text{ and } \mu_\eps =0 \text{ if }\:
      a\le u_\eps \le b \:\text{ a.e.\ in $\D\times \Omega$}.\label{eq:optmureg}
    \end{alignat}
  \end{subequations}
 \end{corollary}

The next result provides a bound for the difference between minimizers
of \eqref{eq:optcon_linear_s} and its $\eps$-regularized version
with objective \eqref{eq:J}.
\begin{lemma} 
  Let $\D$ be bounded, $\eps>0$ and denote by $\bar u$ the solution to
  \eqref{eq:optcon_linear_s}, i.e., the minimizer of $u\mapsto
  \mathcal J(u,0)$, and by $u_\eps$ the minimizer of $u\mapsto
  \mathcal J(u,\eps)$. Then
  \begin{equation}
    \|\bar u - u_\eps\|_\bs{V}^2 \le \eps\beta\alpha^{-1}|\D|,
  \end{equation}
  where $|\D|$ denotes the volume of $\D$.
\end{lemma}
\begin{proof}
  Since $\bar u$ and $u_\eps$ are the unique minimizers of
  \eqref{eq:J} for $\eps=0$ and $\eps>0$, respectively, we have for
  all $v \in \bs U_{\!\text{ad}}$ that
  \begin{alignat}{2}
    \langle (A^{-\star}A^{-1} + \alpha)\bar u, v-\bar u \rangle_\bs{V} + \beta \int_D\! \left(\|v\|_\Omega -
    \|\bar u \|_\Omega \right)\,d\bs x &\ge -\langle g,v-\bar u \rangle_\bs{V},\label{eq:var_optu*}\\
    \langle (A^{-\star}A^{-1} + \alpha)u_\eps, v-u_\eps\rangle_\bs{V} + \beta \int_D \!
    \left(\sqrt{\|v\|^2_\Omega\!+\!\eps^2} - \sqrt{\|u_\eps\|^2_\Omega\!+\!\eps^2}\right)\,d\bs x
    &\ge - \langle g,v-u_\eps \rangle_\bs{V},\label{eq:var_optueps}
  \end{alignat}
  where $g=A^{-\star}(A^{-1}Bm-\hat y_d)$.  Using $v=u_\eps$ in
  \eqref{eq:var_optu*} and $v=\bar u $ in \eqref{eq:var_optueps}, and
  summing the resulting inequalities yields
  \begin{multline} \label{eq:est1}
    \langle (A^{-\star}A^{-1} + \alpha)(u_\eps-\bar u ),u_\eps-\bar u \rangle_\bs{V} \le\\
    \beta \int_D \left(\|u_\eps\|_\Omega -
    \|\bar u \|_\Omega + \sqrt{\|\bar u \|^2_\Omega+\eps^2} -
    \sqrt{\|u_\eps\|^2_\Omega+\eps^2}\right) \,d\bs x.
  \end{multline}
  The expression under the integral on the right hand side can be
  estimated pointwise:
  \begin{alignat*}{2}
    \|u_\eps\|_\Omega - \sqrt{\|u_\eps\|^2_\Omega+\eps^2} -
    \|\bar u \|_\Omega + \sqrt{\|\bar u \|^2_\Omega+\eps^2}
    &\le
    \sqrt{\|\bar u \|^2_\Omega+\eps^2} -\|\bar u \|_\Omega\\
    &\le \frac{\eps^2}
        {\sqrt{\|\bar u \|^2_\Omega+\eps^2} +\|\bar u \|_\Omega}\le \eps.
  \end{alignat*}
  Integrating this estimate over $\D$ and combination with \eqref{eq:est1}
  proves the result.
\end{proof}

In the next two sections, we introduce a first and a second-order
algorithm for the solution of \eqref{eq:J} without bound
constraints on the control. Both algorithms avoid approximation in
random space using sampling and only iterate over functions defined on the
physical space $\D$. Thanks to the linearity of the governing
equation and the Gaussianity of the uncertain parameter, computations
over the (potentially high-dimensional) random space $\Omega$ can be
performed analytically. In practice, these computations can be
performed efficiently using low-rank operator approximations as
proposed in \cref{sec:lowrank}.


\section{Norm reweighting for shared sparsity control problem}
\label{sec:reweighting}
To develop an efficient algorithm for
\eqref{eq:optcon_linear_s}, we make the simplification $\bs
U_{\!\text{ad}} = \bs V = L_\mu^2(\Omega;L^2(\D))$, i.e., we consider a problem
without bound constraints on the control. The algorithm we propose
below cannot be generalized to incorporate inequality constraints.
Such constraints would destroy the
Gaussianity of the controls in the auxiliary problems we introduce below.
We introduce a family of objective functions that are quadratic in $u$,
involve the parameter $\eps\ge 0$ and a {\em weighting} function $\nu:\D\to
\mathbb R$ with $\nu(\bs x)>0$:
\begin{equation}\label{eq:optcon_linear_sreg}
  \bar{\mathcal J}(u,\nu,\eps):= Q(u,m) + \frac{\beta}{2}\int_\D
  \left({\nu}\|u\|_\Omega^2 + \eps^2\nu + \nu^{-1}\right)\,d\bs x.
\end{equation}
Here, the function $\nu$ weights the term $\|u\|^2_\Omega(\bs x)$,
and the latter two terms, which are not present in
\eqref{eq:optcon_linear_s} will be useful to update the
weighting function $\nu$ in the algorithm presented next. 
Norm reweighting, in the context of under-determined problems also
known as {\em iteratively reweighted least
  squares (IRLS)}, is commonly used to compute
finite-dimensional sparse solutions vectors.  Some of
the analysis presented below extends results from
\cite{DaubechiesDeVoreFornasierEtAl10} to infinite dimensions. Similar
generalizations to infinite dimensions, also for the non-convex case,
are presented in \cite{ItoKunisch14, ItoKunisch13}, where the method
is referred to as {\em monotone algorithm}. The basic idea behind
these methods applied to the stochastic shared sparsity control
problem is presented next.

For a given, monotonously decreasing sequence $(\eps_k)_{k\ge 0}$ with
$\eps_k>0$, the algorithm performs alternate minimization of
$\bar{\mathcal J}(u,\nu,\eps)$ with respect to $u$ and $\nu$.  Given
an initialization $\nu^0$, $\nu_0(\bs x)>0$ for all $\bs x\in \D$, we
compute, for $k\ge 1$, a sequence
of iterates $u^k=u^k(\omega,\bs x)$, $\nu^k(\bs x)$ as follows:
\begin{alignat}{2}
u^{k+1} &= \argmin_{u\in \bs V} \bar{\mathcal J}(u,\nu^k,\eps_k),\label{eq:ad1}\\
\nu^{k+1} &= \argmin_{\nu \in L^\infty(\D)} \bar{\mathcal J}(u^{k+1},\nu,\eps_{k+1}). \label{eq:ad2}
\end{alignat}
Let us first discuss the minimization \eqref{eq:ad2}. Taking
variations of \eqref{eq:optcon_linear_sreg} with respect to $\nu$, and
using that $\nu$ must be positive, \eqref{eq:ad2} implies that
\begin{equation}\label{eq:opt_nu}
\nu^{k+1}(\bs x) = \left(\|u^{k+1}\|^2_\Omega(\bs
x)+\eps^2_{k+1}\right)^{-\frac 12}.
\end{equation}
Note that $\nu^{k+1}$ is a function only of $\bs x\in \D$, and, for
each $\bs x$, it requires integration over the random space $\Omega$.

Since \eqref{eq:ad1} is a strictly convex least squares\footnote{The occurrence of this least square
problem is the reason why this algorithm is referred to as iteratively
reweighted least squares (IRLS) method in the literature
\cite{DaubechiesDeVoreFornasierEtAl10}.}  problem, it
has a unique solution $u^{k+1}$. Taking variations with
respect to $u$, one finds that $u^{k+1}$ is characterized by the
optimality condition
\begin{equation}\label{eq:normaleq}
  \left[A^{-\star}A^{-1} + (\alpha+\beta \nu^k)\right]u^{k+1} =
  A^{-\star} (y_d - A^{-1}(f + B m)).
\end{equation}
After introducing state and adjoint variables $y^{k+1},p^{k+1} \in \bs{Y}$, this
is equivalent to
\begin{alignat*}{2}
  A y^{k+1} - u^{k+1} - f - B m &= 0,\\ 
  A^\star p^{k+1} - y_d + y^{k+1} &= 0,\\ 
  -p^{k+1} + (\alpha + \beta \nu^k) u^{k+1} &= 0. 
\end{alignat*}
%
Next, we observe that using the optimality condition \eqref{eq:opt_nu}
in \eqref{eq:optcon_linear_sreg} yields
\begin{equation}\label{eq:barJ}
  \bar{\mathcal J}(u^k,\nu^k,\eps_k)= Q(u^k,m) + \beta\int_\D
  \left(\|u^k\|_\Omega^2 + \eps^2_k\right)^{1/2}\,d\bs x = \mathcal J(u^k,\eps_k),
\end{equation}
shining light onto the relation between $\mathcal J$ and $\bar{\mathcal J}$.
Note, however, that $u^k$ is in general not a minimizer of $u\mapsto
\mathcal J(u,\eps_k)$.

The alternate minimization property \eqref{eq:ad1}, \eqref{eq:ad2}
shows that the following monotonicity holds for $k=0,1,2,\ldots$,
\begin{alignat}{2}\label{eq:Jmonotone}
\bar{\mathcal J}(u^{k+1},\nu^{k+1},\eps_{k+1}) \le \bar{\mathcal J}(u^{k+1},\nu^{k},\eps_{k+1}) \le
\bar{\mathcal J}(u^{k+1},\nu^k,\eps_k) \le \bar{\mathcal J}(u^k,\nu^k,\eps_k).
\end{alignat}
Here, the first inequality follows from the optimality of $\nu^{k+1}$
for \eqref{eq:ad2}, the second inequality from the definition
\eqref{eq:J} and from $\eps_{k+1}\le \eps_k$, and the last inequality
from the optimality of $u^{k+1}$ for \eqref{eq:ad1}.\footnote{This monotonicity
property of the algorithm is the reason why this class of algorithms
are referred to as monotone algorithms in \cite{ItoKunisch14,
  ItoKunisch13}.} The iterates of the algorithm satisfy a boundedness
property summarized in the next lemma.
\begin{lemma}\label{lemma:bounded}
Let $\eps_k$, $k=0,1,\ldots$ be a non-increasing sequence of positive
numbers, and $(u^0,\nu^0)$ a given initialization. Then, the iterates
$u^k$ satisfy
\begin{equation}\label{eq:sum_bounded}
\sum_{k=0}^\infty\int_\Omega  \int_\D (u^k - u^{k+1})^2\,d\bs x\,d\mu
< \infty.
\end{equation}
\end{lemma}
\begin{proof}
The following estimate holds:
\begin{alignat*}{2}
&\:\bar{\mathcal J}(u^{k},\nu^{k},\eps_{k}) - \bar{\mathcal
  J}(u^{k+1},\nu^{k+1},\eps_{k+1})\\
  \ge & \:\bar{\mathcal
  J}(u^{k},\nu^{k},\eps_{k}) - \bar{\mathcal J}(u^{k+1},\nu^{k},\eps_k)\\
  \ge & \: 
  \frac 12 \int_\Omega  \int_\D (A^{-1}(u^k-u^{k+1}))^2 + (\alpha +
  \beta \nu^k) (u^k - u^{k+1})^2\,d\bs x\,d\mu\\
  \ge & \: \frac \alpha 2 \int_\Omega  \int_\D (u^k - u^{k+1})^2\,d\bs x\,d\mu.
\end{alignat*}
Here, the first inequality uses \eqref{eq:Jmonotone}, and the second
inequality follows from a Taylor expansion of $u\mapsto
\bar{\mathcal J}(u,\nu^k,\eps_k)$ at $u^{k+1}$ in the direction
$u^k-u^{k+1}$, in which due to \eqref{eq:ad1} the first-order term
vanishes, i.e. $\langle u^{k+1}, u^k - u^{k + 1} \rangle_\bs{V} = 0$.
Summing the above estimate over $k$ proves
 \eqref{eq:sum_bounded}.
\end{proof}
Note that the above result implies that, in particular,
\begin{equation*}
 \int_\Omega  \int_\D (u^k - u^{k+1})^2\,d\bs x\,d\mu \to 0 \quad
 \text{ as } k\to \infty.
\end{equation*}
However, \cref{lemma:bounded} does not imply convergence of the
IRLS algorithm when $\eps_k\to 0$. The next result provides a
convergence result for the case that $\eps_k\to\bar\eps>0$.
\begin{lemma}\label{lemma:conv_epsbar}
Let $\eps_k, k=1,2,\ldots$ be a non-increasing sequence with
$\lim_{k\to\infty} \eps_k=\bar\eps>0$. Then, for any initialization
$(u^0,\nu^0)$, 
$  u^k \to u_{\bar\eps}$  strongly in $\bs V$ as $k\to \infty$.
\end{lemma}
\begin{proof}
  As above, we denote by $p^k\in \bs V$ the adjoint variable corresponding to
  $u^k$.  We consider the derivative of $\mathcal J$ with respect to
  $u$, $\mathcal
  J_{u}(u^k,\bar \eps) = -p^k + \alpha u^k + \beta \bar\nu^k u^k$, where
  $\bar\nu^k = (\|u^k\|^2_\Omega+\bar\eps^2)^{-1/2}$. Then,
  \begin{equation*}
    \begin{split}
    \|-p^k + \alpha u^k + \beta \bar\nu^k u^k\|_\bs{V} 
    =\|-(p^k-p^{k+1}) + \alpha(u^k-u^{k+1}) + \beta(\bar\nu^ku^k-\nu^ku^{k+1})\|_\bs{V} \\
    \le \|p^k-p^{k+1}\|_\bs{V} + \alpha\|u^k-u^{k+1}\|_\bs{V}
    +\beta\bar\eps^{-1}\|u^k-u^{k+1}\|_\bs{V} \\
 + \beta\bar\eps^{-2}\left\|u^{k+1}\left(\sqrt{\|u^k\|_\Omega^2+\bar\eps^2}- 
 \sqrt{\|u^k\|_\Omega^2+\eps_k^2}\right)\right\|_\bs{V}, 
    \end{split}
  \end{equation*}
  where we used that $-p^{k+1}+\alpha u^{k+1} +
  \nu^k u^{k+1} = 0$ in the first equality, with $u^{k+1}$ is the
  minimizer of $\bar{\mathcal J}(u, \nu^k, \epsilon_k)$, $\nu^k$ 
  defined as \eqref{eq:opt_nu} and the
  assumption $\eps_k\ge \bar\eps$ in the estimation. Using
  \cref{lemma:bounded} and the fact that $\eps_k\to\bar\eps$ implies
  that $J_{u}(u^k,\bar \eps) \to 0$ as $k\to \infty$. Finally, since
  $\|J_{u}(u^k,\bar \eps)\|_\bs{V} \ge \alpha\|u^k-u_{\bar\eps}\|_\bs{V}$, we
  obtain the postulated convergence result.
\end{proof}

\section{Newton method for reweighted shared sparsity control problem}
\label{sec:newton}
Rather than using optimization objectives that depend on the
control $u\in \bs V$, or both on the control $u$ and the weighting
function $\nu$, here we propose a reduced objective that only depends
on $\nu$ (and on $\eps$). This objective considers
$u$ as a function of $\nu$ and thus 
a numerical scheme for this
reduced formulation only requires iterations for $\nu$.
The optimality condition with respect to $u$ in
\eqref{eq:optcon_linear_sreg} shows that
\begin{equation}\label{eq:random_to_physical}
u = S_{\nu} v,
\end{equation}
where $S_\nu:\bs V\to \bs V$ and $v$ are defined as
\begin{equation}\label{eq:operator_S_nu}
  S_{\nu} := \left[A^{-\star}A^{-1} + (\alpha+\beta \nu )\right]^{-1},
  \quad
  v := A^{-\star} (y_d - A^{-1}(f + B m)). 
\end{equation}
Since for every $\nu\in L^\infty(\D)$, $\nu\ge 0$,
\eqref{eq:random_to_physical} has a unique solution $u$, we can
consider $u$ as a function of $\nu$ only, leading to the following
reduced version of \eqref{eq:optcon_linear_sreg},
\begin{align}\label{eq:tildeJ}
  \tilde{\mathcal J}(\nu,\eps)
  &:= Q(S_{\nu} v, m) + \frac{\beta}{2}\int_\D
  \left({\nu}\|S_{\nu} v\|_\Omega^2 + \eps^2\nu + \nu^{-1}\right)\,d\bs x.
\end{align}
Note that this is a non-quadratic functional in
$\nu$. Its derivative 
in a direction $\tilde\nu$ is:
\begin{align*}
\tilde{\mathcal J}_{\nu}(\nu,\eps)(\tilde\nu)
=& \int_\D \left(A^{-\star}(A^{-1} S_{\nu}v + A^{-1} Bm - \hat y_d)
+ \alpha S_{\nu} v + \beta\nu S_\nu v\right) (S_{\nu} v)_{,\nu}
(\tilde \nu)\,d\bs x \\ &+
\frac{\beta}{2} \int_{\D} \left(\|S_{\nu} v\|_\Omega^2 + \eps^2 -
\frac{1}{\nu^2}\right) \tilde \nu \,d \bs x \\
=& \frac{\beta}{2} \int_{\D} \left(\|S_{\nu} v\|_\Omega^2 + \eps^2 - \frac{1}{\nu^2}\right) \tilde \nu \,d \bs x,
\end{align*}
where $(S_{\nu}v)_{,\nu}$ denotes variation of $S_\nu v$ with respect to $\nu$, and
the first term in the second expression vanishes since
$
[A^{-\star} A^{-1} +(\alpha + \beta \nu)]S_{\nu} v - A^{-\star} (\hat
y_d - A^{-1} B m) = 0
$. 
Using the $L^2(\D)$-inner product, the
gradient $\mathcal G$ of $\tilde{\mathcal J}$ with respect to $\nu$ is thus
\begin{equation}\label{eq:G}
\mathcal G(\nu) = \|S_{\nu} v\|_\Omega^2 + \eps^2 - \frac{1}{\nu^2},
\end{equation}
where for simplicity, we neglect to denote the dependence of $\mathcal
G$ on $\eps$.
Note that using \eqref{eq:random_to_physical} and
\eqref{eq:operator_S_nu}, and introducing the control, state and
adjoint variables $u_\eps$, $y_\eps$ and $p_\eps$, the
first-order optimality condition $\mathcal G(\nu)=0$ is equivalent to
the optimality system \eqref{eq:opts1reg}, \eqref{eq:opts2reg},
\eqref{eq:opts3reg} with $\mu_\eps = 0$. This system uniquely
characterizes $u_\eps$, and thus the corresponding $\nu_\eps =
(\|u_\eps\|^2_\Omega + \eps^2)^{1/2}$ satisfies $\mathcal
G(\nu_\eps)=0$ and is thus the unique minimizer of
\eqref{eq:tildeJ}.

A possible choice for an iterative fixed-point method to solve
$\mathcal G(\nu)=0$ with $\eps=\eps_{k+1}>0$ is, for given $\nu^k$,
to compute $\nu^{k+1}$ from
\begin{equation}
 \frac{1}{(\nu^{k+1})^2} = \|S_{\nu^k} v\|_\Omega^2 + \eps_{k+1}^2.
\end{equation}
Taking square roots and reciprocals, we thus rediscover
the IRLS method from \cref{sec:reweighting} as an
iterative fixed point method for solving $\mathcal G(\nu) =
0$. This also implies that the gradient can be computed from the
iterates of the IRLS algorithm as
\begin{equation}\label{eq:grad_reweighting}
\mathcal G(\nu^k) = \frac{1}{(\nu^{k+1})^2} - \frac{1}{(\nu^{k})^2},
\end{equation}
which provides a possible termination criterion for the IRLS algorithm.

Computing second variations of $\tilde{\mathcal J}$ with respect to
$\nu$ in a direction $\tilde \nu$ yields the following Hessian
operator
\begin{equation}\label{eq:Hessian}
\mathcal H(\nu) \tilde \nu=  -  2 \beta \int_\Omega (S_{\nu} v)\odot S_{\nu}
((S_{\nu} v)\odot  \tilde \nu)  \,d\mu +  \frac{2}{\nu^3} \odot \tilde\nu,
\end{equation}
where $v$ as defined in \eqref{eq:operator_S_nu} is a Gaussian random
process, and $\odot$ denotes the pointwise multiplication in space. That
is, for $f,g:\D\to \mathbb R$ and $h:\D\times \Omega\to \mathbb R$,
$(f \odot g)(\bs{x}) := f(\bs{x})g(\bs{x})$ , and $(h \odot g)(\bs{x},\omega) :=
h(\bs{x},\omega)g(\bs{x})$. In the derivation of \eqref{eq:Hessian}, we have
also used that the derivative of $S_{\nu}v$ with respect to $\nu$
satisfies $(S_{\nu} v)_{,\nu}\tilde \nu = - \beta S_{\nu} ((S_{\nu}
v)\odot\tilde \nu)$ for any random draw $v(\omega)$.  Thus, the Newton update step at an iterate
$\nu=\nu^k$ is
\begin{subequations}\label{eq:Newtonstep}
\begin{align}
- 2 \beta\!\!\int_\Omega\!  (S_{\nu^k} v) \odot S_{\nu^k}((S_{\nu^k} v) \odot \delta \nu )\, d\mu +\!  \frac{2}{(\nu^k)^3} \odot \delta \nu\! &=\!
-\|S_{\nu^k} v\|_\Omega^2 - \eps_{k+1}^2 \!+ \!\frac{1}{(\nu^k)^2},\\
\nu^{k+1} &= \nu^k + \delta \nu.
\end{align}
\end{subequations}
Due to the computational cost of integration over the
(possibly high-dimensional) random space, an efficient implementation
of the IRLS algorithm (\cref{sec:reweighting}) and its Newton
variant (\cref{sec:newton}) is
challenging. Hence, we next
propose low-rank operator approximations that make these computations
feasible. We also present estimates for the truncation error of these
approximations and propose a diagonal preconditioner for the Newton
step \eqref{eq:Newtonstep}.

\section{Low-rank operator approximations}\label{sec:lowrank}
%
Since we assume that the (possibly infinite-dimensional) uncertain
parameter $m$ follows a normal distribution,
\eqref{eq:normaleq}, and equivalently \eqref{eq:random_to_physical},
imply that, for given $\nu$, the corresponding control variable also
follows a normal distribution.
To be precise, if the distribution of $m$ is
$\mu=\mathcal N(m_0, \mathcal C_0)$,
then
$ u \sim \mathcal N(u_\nu,\mathcal Q_\nu)$ with
\begin{alignat}{2}\label{eq:ubarCu}
  u_\nu = S_\nu A^{-\star} (y_d - A^{-1}(f + B m_0)), \quad
  \mathcal Q_\nu = S_{\nu}A^{-\star}A^{-1}B\mathcal C_0 B^\star A^{-1}A^{-\star}S_\nu^\star ,
\end{alignat}
where $S_\nu$ is defined as in \eqref{eq:operator_S_nu}.
Here, $S_\nu^\star$ is the adjoint operator of
$S_\nu$ with respect to the $L^2$-inner product, but since $S_\nu$ is
self-adjoint, $S_\nu=S_\nu^\star$.

%
In this section, we develop a method that exploits operator properties
to enable the efficient implementation of the algorithms from
\cref{sec:reweighting,sec:newton}.  In particular, we use
properties that are typical for instance for inverse elliptic PDE
operators, to
construct low-rank operator approximations. Moreover, we provide
estimates for the resulting errors in terms of the truncated
eigenvalues of the low-rank approximations.

\subsection{Spectral decomposition of $\mathbf {A^{-\star}A^{-1}}$ and truncation error analysis}
\label{subsec:truncation}

We make the assumption that the symmetric and positive definite
solution operator $A^{-\star}A^{-1}$ is a trace class operator, and
thus its spectrum is rapidly decaying. This can be explored to enable
fast computations based on low-rank approximations of the operator $S_\nu$,
as discussed next. We denote by $D_\nu := (\alpha +
\beta\nu) I$, and assume we have a spectral decomposition of
$A^{-\star}A^{-1}$ with decreasing eigenvalues $\lambda_i$ and
corresponding eigenvectors $u_i$, $i\ge 1$. Thus,
\begin{equation}
S_\nu = (A^{-\star}A^{-1} + D_\nu)^{-1} =
(U\Lambda U^\star  + D_\nu)^{-1},
\end{equation}
where
$U^\star v = \left(\langle u_i,v\rangle_{L^2(\D)}\right)_{i\ge 1} \in
\ell_2$ for  $v\in L^2(\D)$ and
$U\bs y = \sum_{i=1}^\infty y_iu_i \text{ for } \bs y = (y_i)_{i\ge 1}\in
\ell_2$,
that is, $U$ and $U^\star$ are operators corresponding to a change of basis.
Moreover, $\Lambda$ is a diagonal operator with entries $\lambda_i$.
To approximate $(U\Lambda U^\star + D_\nu)^{-1}$, in
the following, we will truncate the
eigenvalue expansion of $A^{-\star}A^{-1}$. However, first we use the
Sherman-Morrison-Woodbury formula and find
\begin{alignat}{2}\label{eq:SMW}
  S_\nu = (D_\nu + U\Lambda U^\star  )^{-1} 
  &= D_\nu^{-1} - D_\nu^{-1}(U(\Lambda^{-1} + U^\star D_\nu^{-1}U)^{-1} U^\star )D_\nu^{-1}.
\end{alignat}
To show how to control the approximation error resulting from eigenvalue
truncation, we next derive an upper bound for the positive definite
operator $S_\nu$. To compare positive definite operators $E$
and $F$, we say that $E\preccurlyeq F$ if $F-E$ is positive
semidefinite. Then, because $\nu$ is positive, we have $U^\star
D_\nu^{-1}U\preccurlyeq \alpha^{-1}I$,
which implies that
\[
\Lambda^{-1} + \big(U^\star D_\nu^{-1}U\big)
\preccurlyeq \Lambda^{-1} + \alpha^{-1} I = \text{diag}\Big( \frac{\lambda_i + \alpha}{\alpha \lambda_i} \Big).
\]
Consequently,
\[
\left(\Lambda^{-1} + \big(U^\star D_\nu^{-1}U\big)\right)^{-1}
\succcurlyeq \text{diag}\Big(\frac{\alpha\lambda_i}{\lambda_i + \alpha}\Big).
\]
Along with \eqref{eq:SMW}, we conclude
\begin{alignat*}{2}
  (D_\nu + U\Lambda U^\star  )^{-1} \preccurlyeq
  D_\nu^{-1} - D_\nu^{-1}\left(U  \text{diag}\Big(\frac{\alpha\lambda_i}{\lambda_i + \alpha}\Big)   U^\star \right)D_\nu^{-1}.
\end{alignat*}
Let us now consider an approximation of $A^{-\star}A^{-1}$ obtained by
truncation of the eigenvalue expansion after the $r$ largest
eigenvalues.  The corresponding truncated analogues of $U$
and $\Lambda$ are denoted by $U_r$ and $\Lambda_r$, respectively. Then,
\begin{subequations}\label{eq:truncation}
\begin{alignat}{2}
  S_\nu &= (D_\nu + U_r\Lambda_r U_r^\star
  )^{-1} + R, \label{eq:truncation1}\\
   \: \text{where }\quad \displaystyle\frac{\text{Tr}(R)}{\text{Tr}(S_\nu)} &\le  \frac{\sum_{i=r+1}^\infty\frac{\lambda_i}{\lambda_i + \alpha}}{\sum_{i=1}^\infty\frac{\lambda_i}{\lambda_i + \alpha}}.\label{eq:truncation2}
\end{alignat}
\end{subequations}
Here, $\text{Tr}(\cdot)$ denotes the operator trace.
This shows that the contribution to the truncation error is small for
eigenvalues that are small compared to $\alpha$. Note that the
truncation error depends on $r$. This truncation error can be made arbitrary small
by choosing $r$ large enough.
This is of
practical importance as it provides guidance on where to truncate the
eigenvalue expansion. We obtain the following approximation
$S_{\nu,r}$ of $S_\nu$ as in \eqref{eq:SMW}.
\begin{alignat}{2}\label{eq:Strunc}
  S_{\nu,r}:=(D_\nu + U_r\Lambda_r U_r^\star)^{-1}
  = D_\nu^{-1} - D_\nu^{-1}(U_r(\Lambda_r^{-1} + U_r^\star D_\nu^{-1}U_r)^{-1} U_r^\star )D_\nu^{-1}.
\end{alignat}
%
Note that the proposed algorithms will only require application of
this operator to vectors, which can be done efficiently as will be discussed
in \cref{subsec:compcost}.  To summarize, for given $\nu$, the
corresponding optimal controls are normally distributed.  Given a
truncated eigenvalue expansion of $A^{-\star}A^{-1}$, this
distribution can be approximated replacing $S_\nu$ by $S_{\nu,r}$ in
\eqref{eq:ubarCu}. In the remainder of this section, we derive
analogues for the gradient and Hessian of $\tilde {\mathcal J}$
building on the approximation $S_{\nu,r}$.

\subsection{Gradient and IRLS using low-rank approximation}\label{subsec:comp_nu}
After characterizing the distribution of the optimal controls $u$,
both the IRLS algorithm and its Newton variant require computation of
$\|u\|_\Omega$, which
involves integration over random space.
To be precise, $\|u\|_\Omega = \|u\|_\Omega(\bs x)$ is an integration
over the Gaussian distribution $\mathcal N(u_\nu,\mathcal
Q_\nu)$ with mean and covariance defined in \eqref{eq:ubarCu}, where
$S_\nu$ is replaced by $S_{\nu,r}$.

To approximate integration over the random space, we use the square
root of $\mathcal Q$ given by
$ \mathcal Q^{1/2} = S_{\nu}A^{-\star}A^{-1}B\mathcal C_0^{1/2}.$
Typical properties of $\mathcal C_0$ and $B$, as well as the trace
class property of $A^{-\star}A^{-1}$ facilitate the approximation of
$A^{-\star}A^{-1}B\mathcal C_0^{1/2}$ with rank-$\tilde r$ operators
$E_{\tilde r}F^\star_{\tilde r} = [e_1,\ldots,e_{\tilde r}]
[f_1,\ldots,f_{\tilde r}]^\star$ as follows:
\begin{equation}\label{eq:E}
  E_{\tilde r}F^\star_{\tilde r} \approx A^{-\star}A^{-1}B\mathcal C_0^{1/2}.
\end{equation}
Here, $e_i\in L^2(\D)$ and $f_i\in \mathscr H$ which we can choose
such that $f_1,\ldots,f_{\tilde r}$ are orthonormal.
This results in the approximation
\begin{equation}
  \mathcal Q = \mathcal{Q}^{1/2} (\mathcal{Q}^{1/2})^\star \approx S_{\nu,r}E_{\tilde r} E^\star_{\tilde r} S_{\nu,r}^\star.
\end{equation}
Thus, $\|u\|^2_\Omega$ can be approximated by
$u_\nu(\bs x)^2 + \sum_{i=1}^{\tilde r} (S_{\nu,r}
e_i)(\bs x)^2$ for $\bs x\in \D$.
Defining
\begin{equation}\label{eq:e0}
  e_0 := A^{-\star} (y_d - A^{-1}(f + B m_0))
\end{equation}
and recognizing that $u_\nu=S_{\nu,r}e_0$ allows the more compact notation
\begin{equation}\label{eq:variation_appro}
  \|u\|^2_\Omega(\bs x) \approx  \|u\|^2_{\Omega,r}(\bs x) := \sum_{i=0}^{\tilde r} (S_{\nu,r} e_i)(\bs{x})^2.
\end{equation}
%
Using this in \eqref{eq:G}, we find the following approximation
$\mathcal G_r(\nu)$ of the gradient of $\tilde{\mathcal J}$:
\begin{equation}\label{eq:Gred}
  \mathcal G(\nu) \approx \mathcal G_r(\nu) :=
  \sum_{i=0}^{\tilde r} (S_{\nu,r}e_i)^2 +
\eps^2 - \frac{1}{\nu^2}
\end{equation}
and a very similar expression to update the reweighting function in
the IRLS method.

\subsection{Hessian using low-rank approximation}
We now derive expressions for the application of the Hessian corresponding to the
gradient $\mathcal G_r(\nu)$ to vectors.
%
We do this by taking derivatives of
$\mathcal G_r(\nu)$ with respect to $\nu$ in a direction $\delta
\nu$. This results in the following Hessian $\mathcal H_r$ based on the
low-rank approximation,
\begin{align*}
  \mathcal H(\nu) \delta \nu \approx
  \mathcal H_r(\nu) \delta \nu&:=  - 2 \beta
  \sum_{i=0}^{\tilde r} (S_{\nu, r} e_i) \odot S_{\nu, r}
((S_{\nu} e_i) \odot \delta \nu) +  \frac{2}{\nu^3} \odot \delta \nu,
\end{align*}
where we have used the identity
$(S_{\nu,r} w)_{,\nu}\tilde\nu = - \beta S_{\nu,r} ((S_{\nu,r} w)\odot\tilde \nu)$.
To summarize, a Newton step based on the low-rank approximation of
$A^{-\star}A^{-1}$ is as follows:
\begin{subequations}\label{eq:newton_step_appro}
\begin{align}
\mathcal H_r(\nu^k)\delta\nu  &= - \mathcal G_r(\nu^k),\\
\nu^{k + 1} &= \nu^{k} + \delta \nu.
\end{align}
\end{subequations}
Despite the low-rank approximation, $\mathcal H_r$ is usually not
explicitly available. Hence, this Newton system must be solved using
an iterative method that, such as the conjugate gradient method,
only requires the application of $\mathcal H_r(\nu^k)$ to vectors.

\subsection{Preconditioning of Newton system}
The convergence of the conjugate gradient (CG) method in each Newton step
depends crucially on the availability of an effective preconditioner.
This is particularly true if the Hessian operator is very
ill-conditioned, as is the case due to the $1/\nu^3$ term in $\mathcal
H_r$, which can vary over many orders of magnitude if
$\eps$ is small. Using the low-rank approximations established
above, we propose a diagonal preconditioner that is effective in
practice, as we illustrate numerically in \cref{sec:numerics}.
The diagonal of the Hessian $\mathcal H_r(\nu)$ is
given by
\begin{equation}\label{eq:precond}
\mathcal P_\diag(\nu) = -2\beta \diag(S_{\nu, r})\sum_{i=1}^{\tilde r}\left((S_{\nu, r}e_i) \odot(S_{\nu, r}e_i)\right) + \frac{2}{(\nu)^3}.
\end{equation}
Here, $\diag(S_{\nu,r})$ is the diagonal of $S_{\nu,r}$, which can be
computed from \eqref{eq:Strunc} using that
\begin{equation}\label{eq:precon_diag}
 \diag(U_r(\Lambda_r^{-1} + U_r^\star D_\nu^{-1}U_r)^{-1} U_r^\star )
 = \sum_{i=1}^r u_i \odot w_i,
\end{equation}
where $w_i=(\Lambda_r^{-1} + U_r^\star D_\nu^{-1}U_r)^{-1}u_i$. Note
 that the terms $S_{\nu, r}e_i$ in \eqref{eq:precond}
are already available from the gradient computation. Moreover, this
diagonal preconditioner depends on $\nu$, which means it must be
recomputed for each Newton step.

\section{Offline-online algorithms for shared sparsity control problem}
The algorithms presented in this section are the result of combining the norm reweighting
algorithms from \cref{sec:reweighting,sec:newton} with the
low-rank approximations from \cref{sec:lowrank}. For large-scale and thus
computationally challenging problems, our method can be split into
an {\em offline phase}, itself consisting of a setup and a compute step,
and an {\em online phase}. The offline phase includes a setup step, in which
we construct a low-rank approximation for the PDE-solution operator
$A^{-\star}A^{-1}$ and the operator $A^{-\star}A^{-1}B\mathcal
C_0^{1/2}$.
This is followed by the offline compute step, in which we solve the
optimization problem \eqref{eq:optcon_linear_s} or \eqref{eq:J}. In
this step,
one can adjust the weight $\beta>0$ for the sparsity-enhancing
term in the objective to obtain the desired sparsity structure which,
in applications, depends on the availability of control devices. In
the online phase, the goal is to compute the optimal control for a
specific (and known) realization of the uncertain parameter
$m(\omega)$. Here, one can use the low-rank approximation for the fast
computation of the optimal control for a specific event $\omega$. In
this step, the sparsity structure and operator approximations
determined in the offline phase are used.
\subsection{Offline phase}
In the offline phase, we first compute a rank-$r$ approximation of the
positive self-adjoint operator $A^{-\star}A^{-1}$, i.e.,
\begin{equation}\label{eq:expansion}
A^{-\star}A^{-1} \approx U_r\Lambda_r U_r^\star .
\end{equation}
As shown in the previous section, the error in the optimal control
solution due to truncation is small when the truncated eigenvalues are
small compared to $\alpha$---see \eqref{eq:truncation}. The low-rank
approximation can be found using either the Lanczos method \cite{Saad03}, or
a randomized algorithm \cite{HalkoMartinssonTropp11}. These methods
only require the application of the linear operator $A^{-*}A^{-1}$ to
vectors, i.e., each application amounts to a solve with the forward
and the adjoint PDE operators $A$ and $A^\star$.
Next, we compute $e_0$ according to \eqref{eq:e0}
and compute a rank-$\tilde r$ approximation of
$A^{-\star}A^{-1}B\mathcal C_0^{1/2}$ as follows:
\begin{equation}\label{eq:EF}
E_{\tilde r}F^\star_{\tilde r} \approx A^{-\star}A^{-1}B\mathcal C_0^{1/2},
\end{equation}
where one can use the low-rank approximation
\eqref{eq:expansion}. In most practical applications, $B\mathcal
C_0^{1/2}$ has a fast decaying spectrum since $\mathcal C_0$ is a
trace class operator and thus, typically, $\tilde r< r$.
After these
preparations, we are ready to either employ the IRLS algorithm
(\cref{alg:reweighting}) or its Newton variant
(\cref{alg:newton}).
%
%
\begin{algorithm}[ht]
\caption{Norm-reweighting for shared sparsity control (IRLS). \label{alg:reweighting}}
\begin{algorithmic}[1]
  \STATE{Setup step: Compute \cref{eq:expansion,eq:e0,eq:EF}.}
  \STATE{Choose $\nu^0\in L^\infty(\D)$,
  $\nu^0(\bs x)>0$ a.e.\ in $\D$, and $\eps_0>0$.}
\FOR{$k=0,1,2,\ldots$ compute}
\STATE\label{reweight:line3}{For $\nu:=\nu^k$, define $S_{\nu,r}$ as
  in \eqref{eq:Strunc}.}
\STATE{Compute $\nu^{k+1} = \left(\sum_{i=0}^{\tilde r} (
  S_{\nu, r}e_i)^2+\eps^2_{k+1}\right)^{-1/2}$. }
\STATE\label{reweight:terminate}{Terminate if the
  norm of $(\nu^{k+1})^{-2} - (\nu^k)^{-2}$ is small.}
\STATE{Update $\eps_{k+1}\le \eps_k$.}
\ENDFOR
\RETURN $\bar\nu:=\nu^{k+1}, u_{\bar \nu}, S_{\bar \nu,r}$.
\end{algorithmic}
\end{algorithm}
%
Note that in \cref{alg:reweighting}, the expression in the termination
criterion (\cref{reweight:terminate}) is the norm of the
reduced gradient \eqref{eq:Gred} due to the relation \eqref{eq:grad_reweighting}.

\begin{algorithm}[ht]
\caption{Newton-CG norm-reweighting for shared sparsity control
  (NIRLS). \label{alg:newton}}
\begin{algorithmic}[1]
  \STATE{Setup step: Compute \cref{eq:expansion,eq:e0,eq:EF}.}
  \STATE{Choose $\nu^0\in L^\infty(\D)$,
  $\nu^0(\bs x)>0$ a.e.\ in $\D$, and $\eps_0>0$.}
\FOR{$k=0,1,2,\ldots$ compute}
\STATE\label{newton:line3}{For $\nu:=\nu^k$, define $S_{\nu,r}$ as
  in \eqref{eq:Strunc}.}
\STATE{Compute $\mathcal G_r(\nu^k)$ according to \eqref{eq:Gred}.}
\STATE\label{newton:terminate}{Terminate if norm of the gradient is small.}
\STATE{Perform CG iterations for Newton system
  \eqref{eq:newton_step_appro} using preconditioner \eqref{eq:precond}.}
\STATE{Update $\eps_{k+1}\le \eps_k$.}
\ENDFOR
\RETURN $\bar\nu:=\nu^{k+1}, u_{\bar \nu}, S_{\bar \nu,r}$.
\end{algorithmic}
\end{algorithm}

\subsection{Online phase}
During the online phase, we compute the optimal control for a specific
realization of the uncertain parameter $m(\omega)$. This step uses the
weight function $\bar\nu$ found in the offline phase, and also uses
the truncated spectral expansion \cref{eq:truncation}.
To be precise, for a sample draw $m(\hat\omega) = m_0+\hat m(\omega)$
from $\mathcal N(m_0, \mathcal C_0)$, the corresponding optimal
control $\hat u \in V$
is computed as
\begin{equation}\label{eq:u_hat}
\hat u = S_{\bar\nu,r} \left(e_0 -  U_r\Lambda_r U^\star_r B\hat
m\right),
\end{equation}
where $\bar\nu$ and $S_{\bar\nu,r}$ are as returned by
\cref{alg:reweighting} or \cref{alg:newton}. Note that
the optimal control $\hat u$ has the sparsity structure determined in
the offline phase.


\subsection{Computational cost}\label{subsec:compcost}
Here, we summarize and compare the dominant computational cost of the
proposed algorithms. We denote by $N$ the discretization dimension of
the state and control variable, and discuss the complexity of the
offline phase (setup and optimization steps) and the online phase
(computation of optimal control). In the arguments below, we assume
that $\tilde r\le r\ll N$.
%
\paragraph{Offline phase: Setup}

In the offline phase, we first compute the truncated spectral
approximation \eqref{eq:expansion}. This requires solves with $A$ and
the adjoint $A^\star$. The number of required solves depends on the
spectrum of the operator $A^{-\star}A^{-1}$, on the value of $\alpha$
and on the truncation error one is willing to commit. The reason why
we report the complexity in terms of PDE solves is that the cost in
terms of operations depends on whether $A$ is available as assembled
matrix, which solvers are applicable to solve systems with
$A$ and $A^*$, and which solvers are available to a user.  To compute a rank-$r$
approximation usually requires $r+d$ products with
$A^{-\star}A^{-1}$, where $d$ is small (e.g., 10) to obtain a
accurate rank-$r$ approximation. If the low-rank approximation is
computed using a randomized singular value decomposition, then $d$ is
the oversampling factor \cite{HalkoMartinssonTropp11}. In the Lanczos
method, adding $d$ iterations enriches the Krylov space and thus leads
to improved accuracy of the dominant directions.
%
It remains to estimate the computational work for computing a singular
value decomposition of
$A^{-\star}A^{-1}B\mathcal C_0^{1/2}$, required to find $E_{\tilde
  r}=[e_1,\ldots,e_{\tilde r}]$, during the offline phase. This step
can build on the truncated spectral decomposition of
$A^{-\star}A^{-1}$ and thus does not require
additional PDE solves.  Hence, the complexity of the offline setup
phase is $2(r+d)$ PDE solves.


\paragraph{Offline phase: Optimization}
After the above setup step, the remaining steps do not require further
PDE solves and we simply estimate the complexity of the proposed
algorithms in terms of elementary linear algebra operations.  Let us
first consider the computations required in each iteration of IRLS
(\cref{alg:reweighting}).  Note that each step is equivalent to
computing
\eqref{eq:G}, the gradient $\mathcal G_r$ of the reduced objective
$\tilde J$. Thus, the computational complexity of one IRLS step
coincides with computing the right hand side for the Newton step
\eqref{eq:Newtonstep}.

Computing $\mathcal G_r$ requires application of the operator $S_{\nu,
  {r}}$, defined in \eqref{eq:Strunc}, to vectors. First, this
necessitates the inverse of the $r\times r$ matrix $(\Lambda_r^{-1} +
U_r^*D_{\nu}^{-1}U_r)$. This step is dominated by the computation of
$U_r^*D_{\nu}^{-1}U_r$, which amounts to $r^2N$ operations.  Since we
assume that $N\gg r$, this dominates computation of an $r\times r$
matrix inverse.  Each application of $S_{\nu,r}$ to a vector requires
$2rN$ operations, amounting overall to a complexity of $2rN(\tilde
r+1)$ operations to compute $\|u\|_\Omega$.  Thus, for the IRLS
algorithm, the computational complexity per iteration is
$rN(r + 2\tilde r)$.

Additionally to the computation of the gradient $\mathcal G_r$,
each iteration of the NIRLS method (\cref{alg:newton}) requires the
application of the Hessian to one vector in each CG step, amounting to
$2\tilde r rN$ operations.  It also
requires to setup the preconditioner matrix, which requires $r^2 N$
operations as can be seen from \eqref{eq:precon_diag}. Hence, we find
that the computational complexity for one inexact Newton-CG step is
$2rN(r + \tilde r + \tilde rn_{\text{cg}} )$, 
where $n_{\text{cg}}$ denotes the number of CG
iterations.

Note that it depends on $r,\tilde r$ how much larger
the complexity of a NIRLS iteration is than an IRLS iteration. If $\tilde r$ is significantly smaller than $r$, as in the
example problems in \cref{sec:numerics}, one CG step only amounts to a
fraction of the complexity of one IRLS
step. Finally, note that all steps in the offline optimization
algorithms have optimal complexity, i.e., they depend linearly on the
discretization dimension $N$.

\paragraph{Online phase}
In the online phase, the only necessary computation is
\eqref{eq:u_hat}. For any fixed sample draw $\hat
m(\omega)$, this is a deterministic expression. This online step
requires multiplication of a vector with $U_r$ and its transpose,
and it requires one application of $S_{\nu, r}$ to a vector. This
amounts to overall about $4rN$ operations, and can thus be
done fast and potentially in real-time, depending on the application.


\section{Numerical experiments}
\label{sec:numerics}
We end this paper with a numerical study for the stochastic control
problem with shared sparsity. Our aims are to study the qualitative
effect of the shared sparsity term on the optimal controls (\cref{sec:qualitative}), and to
investigate the performance and accuracy of the proposed
algorithms (\cref{sec:performance}).

For this purpose, we use three model problems, which all use the
physical domain $\D=(0,1)^2\subset \mathbb R^2$. The boundary is split
in $\partial D_2=\{0\}\times [0,1]$ and $\partial D_1=\partial
D\setminus \partial D_2$. A simple finite difference approximation
(i.e., the five-point stencil) on a mesh of $n\times n$
points is used to discretize the Laplacian
that is part of the differential operator $A$.
While both algorithms we propose allow for a
decreasing sequence of positive values $\eps_1\ge \eps_2\ge \ldots$,
we fix $\eps$ to a small value in our tests, and study the influence
of that value on the performance of the methods.

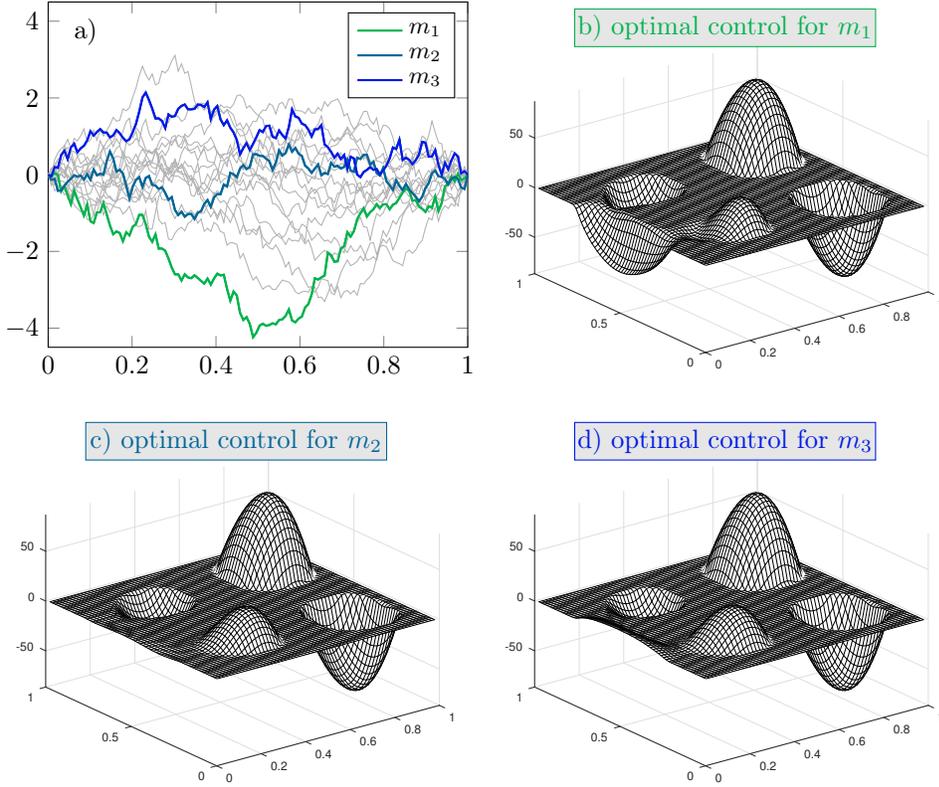
\begin{figure}[tb]\centering
  \begin{tikzpicture}
    \begin{scope}[xshift=-2.5cm]
    \begin{axis}[width=.55\columnwidth,xmin=0,xmax=1,ymin=-4.5,ymax=4.5,compat=1.3, legend pos=north east]
      \addplot[color=blue!30!green,mark=none,thick] table [x=x,y=m1]{data_paper/Fig_1/samples.txt};
      \addlegendentry{\small $m_1$}
      \addplot[color=blue!60!green,mark=none,thick] table [x=x,y=m2]{data_paper/Fig_1/samples.txt};
      \addlegendentry{\small $m_2$}
      \addplot[color=blue!90!green,mark=none,thick] table [x=x,y=m3]{data_paper/Fig_1/samples.txt};
      \addlegendentry{\small $m_3$}
      \foreach \nn in {6,...,20}
      {
        \addplot[color=black!30!white,mark=none,thin] table
        [x=x,y=n\nn]{data_paper/Fig_1/moresamples.txt};
      }
      \addplot[color=blue!30!green,mark=none,thick] table [x=x,y=m1]{data_paper/Fig_1/samples.txt};
      \addplot[color=blue!60!green,mark=none,thick] table [x=x,y=m2]{data_paper/Fig_1/samples.txt};
      \addplot[color=blue!90!green,mark=none,thick] table [x=x,y=m3]{data_paper/Fig_1/samples.txt};
    \end{axis}
    \node[color=black] at (0.5,4.2cm) {a)};
\end{scope}
\node at (6.5,2cm) {\includegraphics[width=0.45\columnwidth]{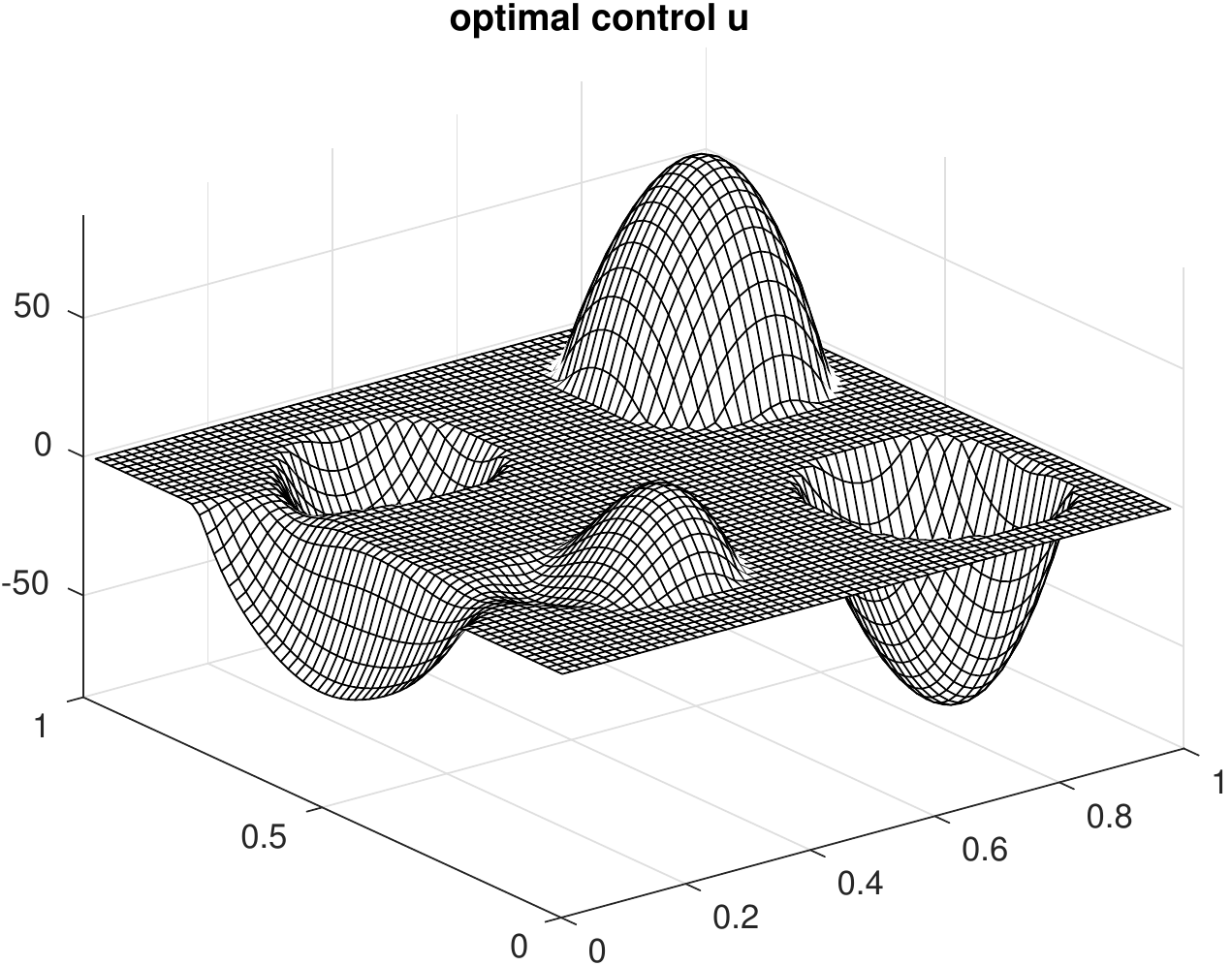}};
\node at (0,-3.5cm) {\includegraphics[width=0.45\columnwidth]{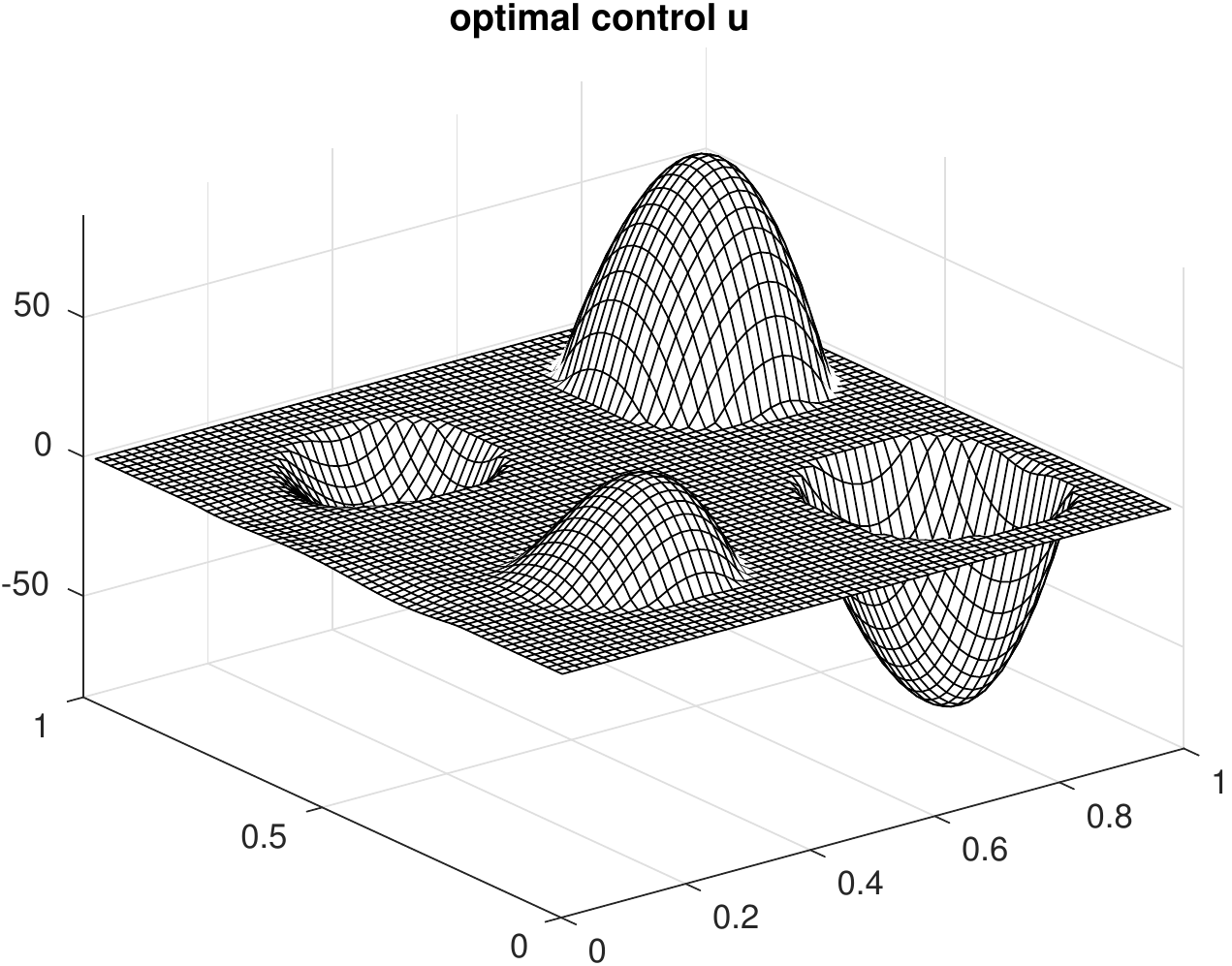}};
\node at (6.5,-3.5cm) (n4) {\includegraphics[width=0.45\columnwidth]{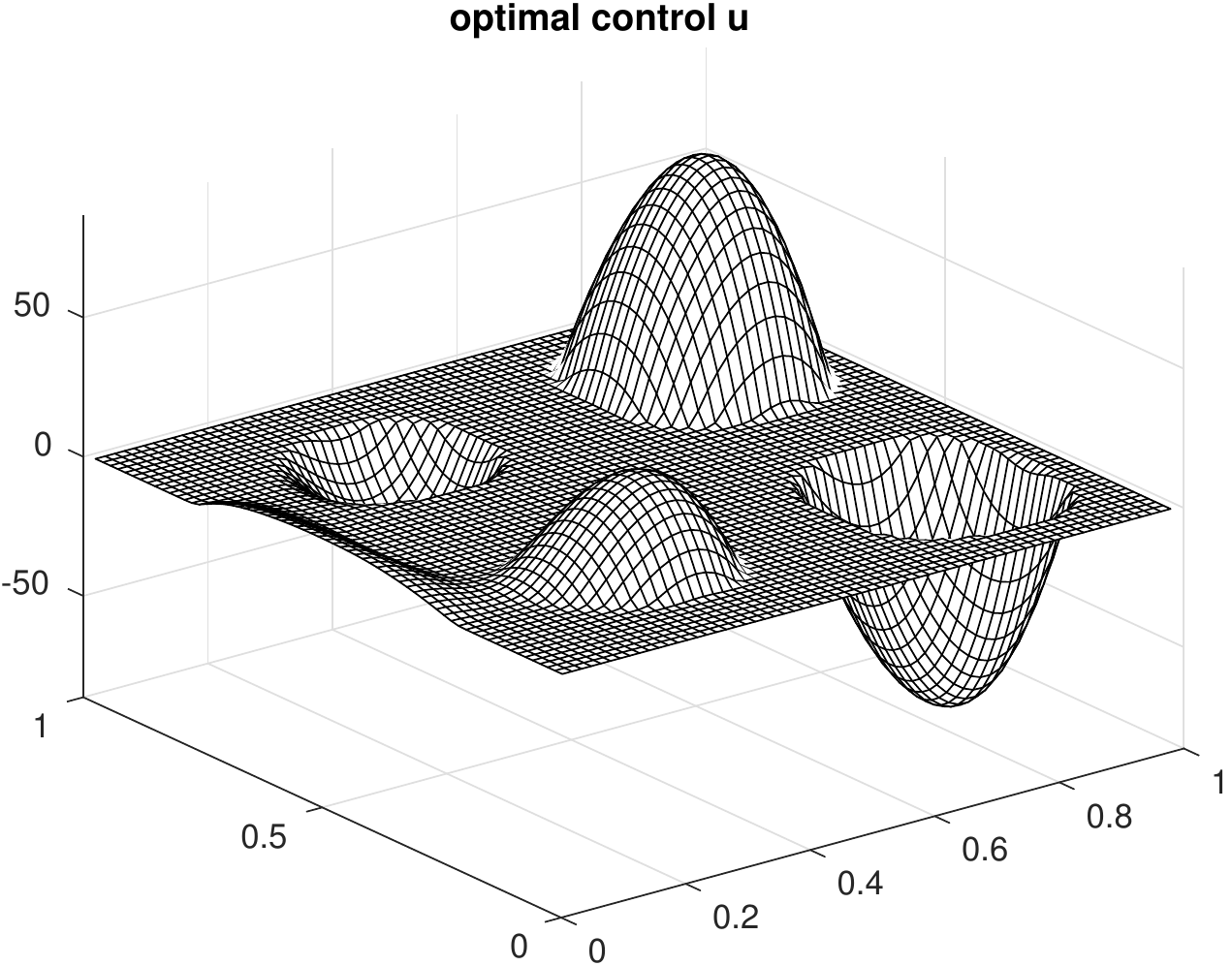}};
\draw [color=blue!30!green,fill=white!90!black] (4.5,4cm) rectangle (8.5,4.5)
node[color=blue!30!green,pos=0.5] {b) optimal control for $m_1$};
\draw [color=blue!60!green, fill=white!90!black] (-2,-1.5cm) rectangle (2,-1cm)
node[color=blue!60!green,pos=0.5] {c) optimal control for $m_2$};
\draw [color=blue!90!green, fill=white!90!black] (4.5,-1.5cm) rectangle (8.5,-1cm)
node[color=blue!90!green,pos=0.5] {d) optimal control for $m_3$};
\end{tikzpicture}
\caption{Results for \cref{ex1}: Random draws for boundary Neumann data (a). The
  highlighted samples are used to compute the optimal controls shown
  in (b), (c) and (d). Note that the controls are different but share
  the same sparsity structure.\label{fig:ex1}}
\end{figure}

\begin{problem}\label{ex1}
  This first problem is of the form of \cref{ex1:poisson} with $k=0$
  and $a(\bs x) \equiv 1$.  Except for the uncertain Neumann boundary
  data, it coincides with Example~1 from \cite{Stadler09}. In
  particular, the PDE-operator is $A=-\Delta$ with zero Dirichlet
  boundary conditions on $\partial D_1$ and Neumann boundary
  conditions on $\partial D_2$. Further, $y_d = \sin(2\pi x)\sin(2\pi
  y)\exp(2x)/6$, $f\equiv 0$, $\alpha=10^{-5}$ and $\beta=10^{-3}$.
  The uncertain parameter field enters as Neumann data on $\partial
  D_2$. This data follows an infinite-dimensional Gaussian
  distribution with mean $m_0 \equiv 0$. The covariance operator is given
  as the inverse elliptic PDE operator $\mathcal C_0 =
  \gamma(-\partial_{{\bs x\bs x}})^{-1}$, with homogeneous Dirichlet
  boundary conditions at the boundary of $\partial D_2$, i.e., at the
  two points $(0,0)$ and $(1,0)$, and with $\gamma=4$. It can easily
  be verified that $\mathcal C_0$ is a symmetric and positive definite
  trace-class operator, and thus defines a valid covariance operator
  \cite{DaPratoZabczyk14}.  Random draws from this distribution are
  shown in \cref{fig:ex1}a, and optimal controls in the remaining
  figures in \cref{fig:ex1}.
\end{problem}

\begin{problem}\label{ex1.5}
  This problem has the form of \cref{ex1:poisson_rhs}. The data are as
  in \cref{ex1}, but the uncertainty enters on the right hand side of
  the equation rather than as Neumann boundary data, and
  $\alpha=5\times 10^{-5}$. The uncertain parameter $m$ is distributed
  as an infinite-dimensional Gaussian random field over the
  two-dimensional physical domain $\D$. Its mean is $m_0 \equiv 0$ and
  the covariance operator is the squared inverse elliptic PDE operator
  $\mathcal C_0 = \gamma(-\Delta)^{-2}$, where $\gamma=20^2$ and the
  Laplace operator $\Delta$ in $\mathcal C_0$ satisfies homogeneous
  Dirichlet conditions on $\{1\}\times [0,1]$ and homogeneous Neumann
  conditions for the remaining boundaries.  $\mathcal
  C_0$ is a valid covariance operator on $L^2(\D)$ as it is symmetric,
  positive and trace-class \cite{DaPratoZabczyk14}. A random draw from
  this distribution and the corresponding optimal control are shown in
  \cref{fig:ex1.5}.
\end{problem}

\begin{problem}\label{ex2}
  This problem has the form of \cref{ex1:helmholtz}. In particular,
  $A=-\Delta - \kappa^2I$ with $\kappa=12$ is the indefinite Helmholtz
  operator.  Moreover, $f\equiv 0$ and $y_d \equiv 0$, i.e., our aim
  is to dampen the uncertain Neumann boundary forcing, whose
  distribution is as in \cref{ex1}. Optimal
  controls for $\alpha=5\times 10^{-5}$ and $\beta=5\times 10^{-4}$ are shown in
  \cref{fig:ex2}. This problem is a substantially simplified version of the
  earthquake engineering/vibration damping problem given as example in
  the introduction, where one aims to find controller locations that
  are best at actively dampening waves originating from uncertain boundary
  forcing. Clearly, this example only uses a single frequency and a
  simple model for wave propagation.
\end{problem}

\begin{figure}[tb]\centering
\begin{tikzpicture}
\node at (0,2cm){
\includegraphics[width=0.45\columnwidth]{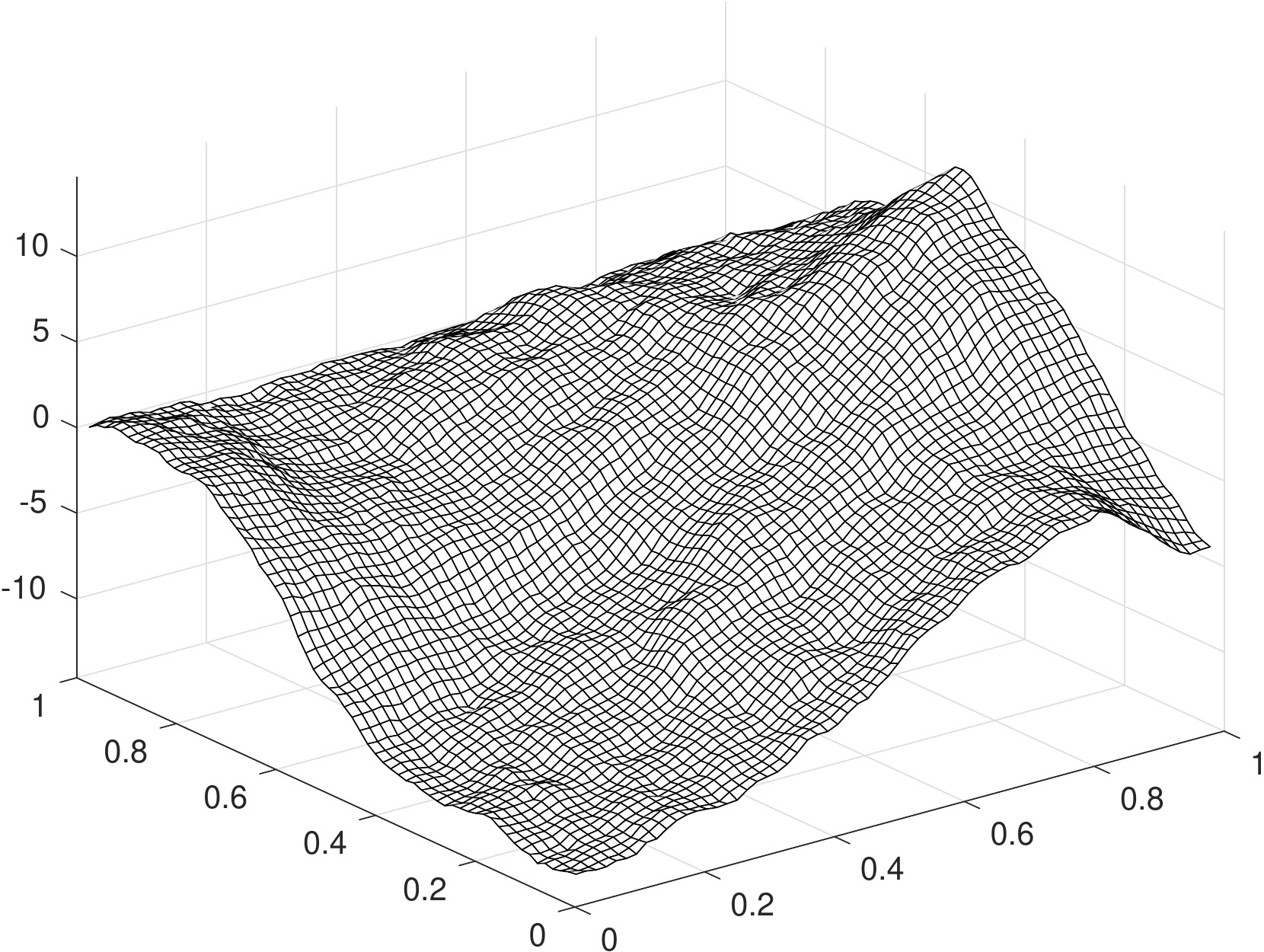}};
\node at (6.5,2cm){
\includegraphics[width=0.45\columnwidth]{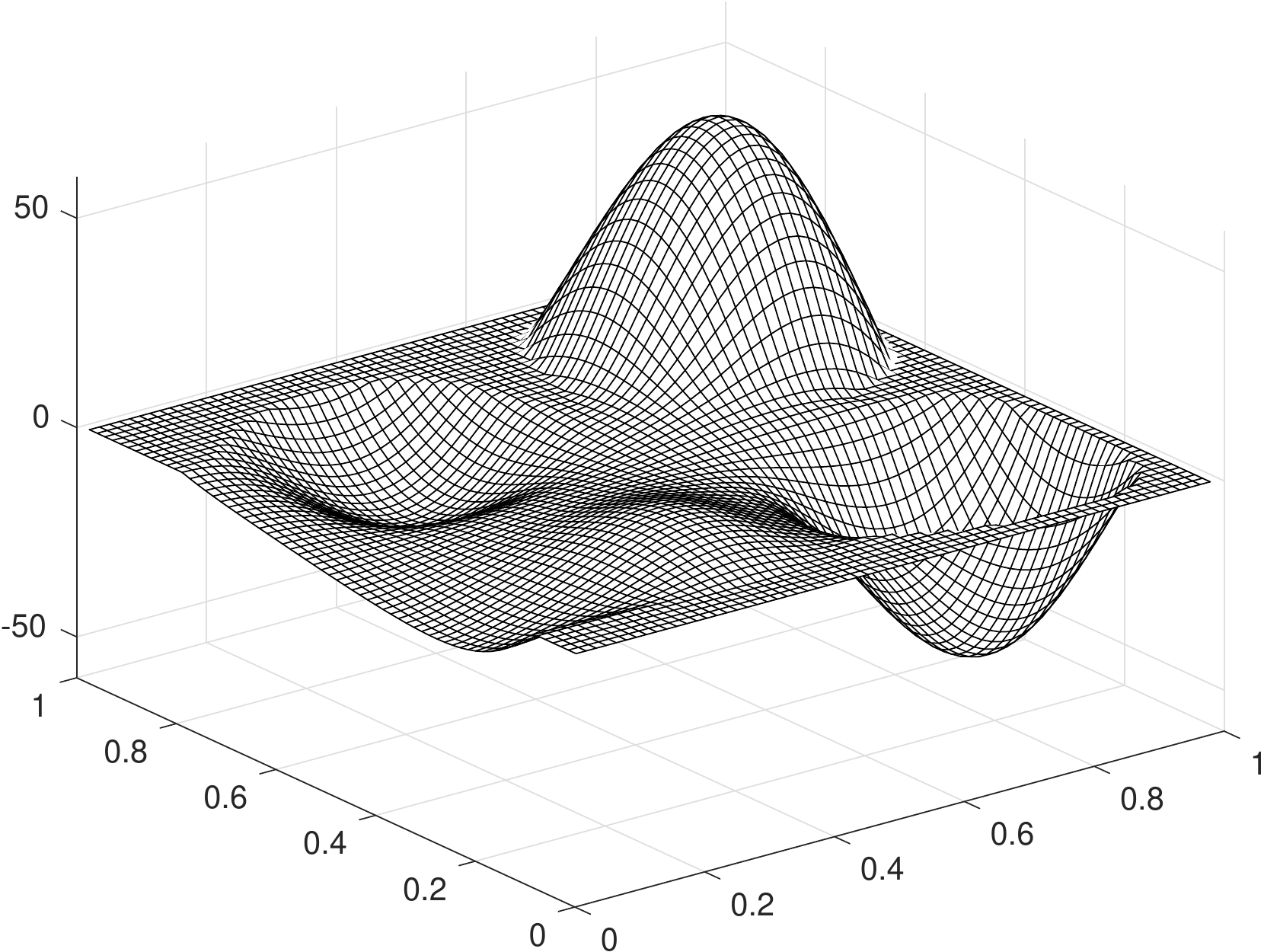}};
\draw [color=blue!50!green,fill=white!90!black] (-1,4cm) rectangle (2,4.5)
node[color=blue!50!green,pos=0.5] {a) random draw $m$};
\draw [color=blue!50!green,fill=white!90!black] (4.5,4cm) rectangle (8.5,4.5)
node[color=blue!50!green,pos=0.5] {b) optimal control for $m$};
\end{tikzpicture}
\caption{Results for \cref{ex1.5}: Shown in (a) is a random draw from
  the Gaussian random field defined over $\D$. Note that all draws of
  the random field satisfy a homogeneous Dirichlet condition on part
  of the boundary.  The figure (b) shows the corresponding optimal
  control.\label{fig:ex1.5}}
\end{figure}

\begin{figure}[tb]\centering
\begin{tikzpicture}
\node at (0,2cm){
\frame{\includegraphics[width=0.35\columnwidth, angle=-90]{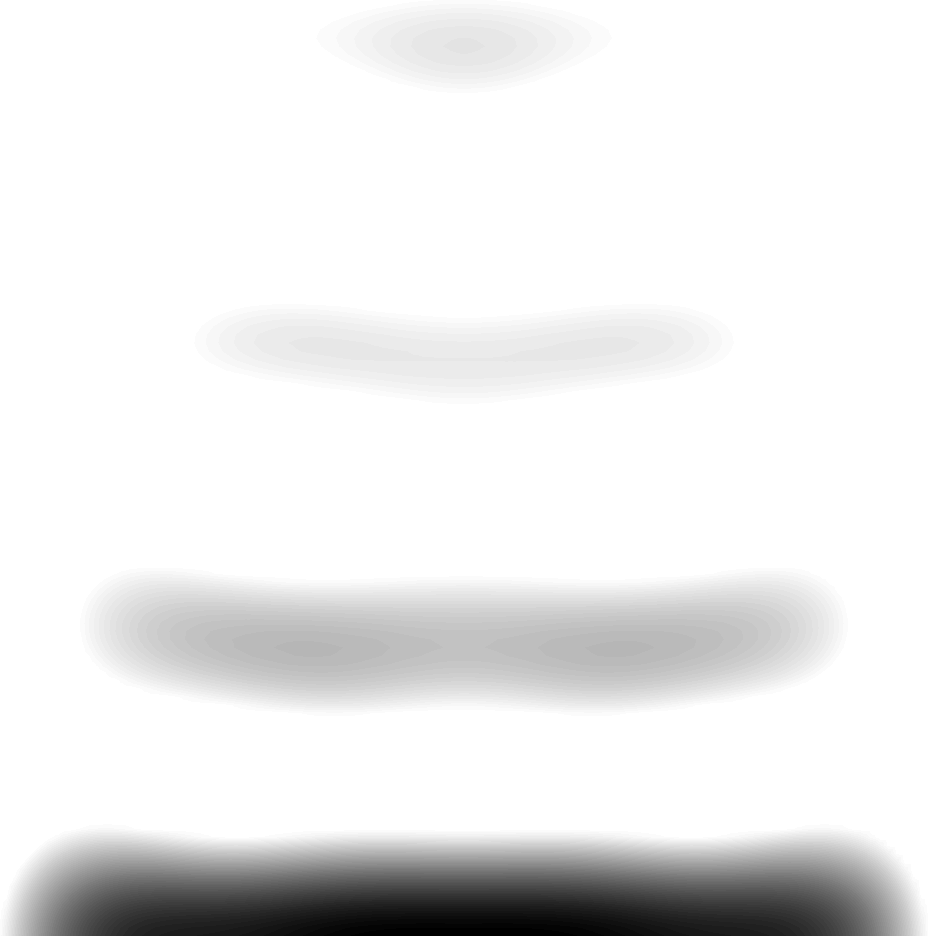}}};
\node at (6.5,2cm){
\includegraphics[width=0.45\columnwidth]{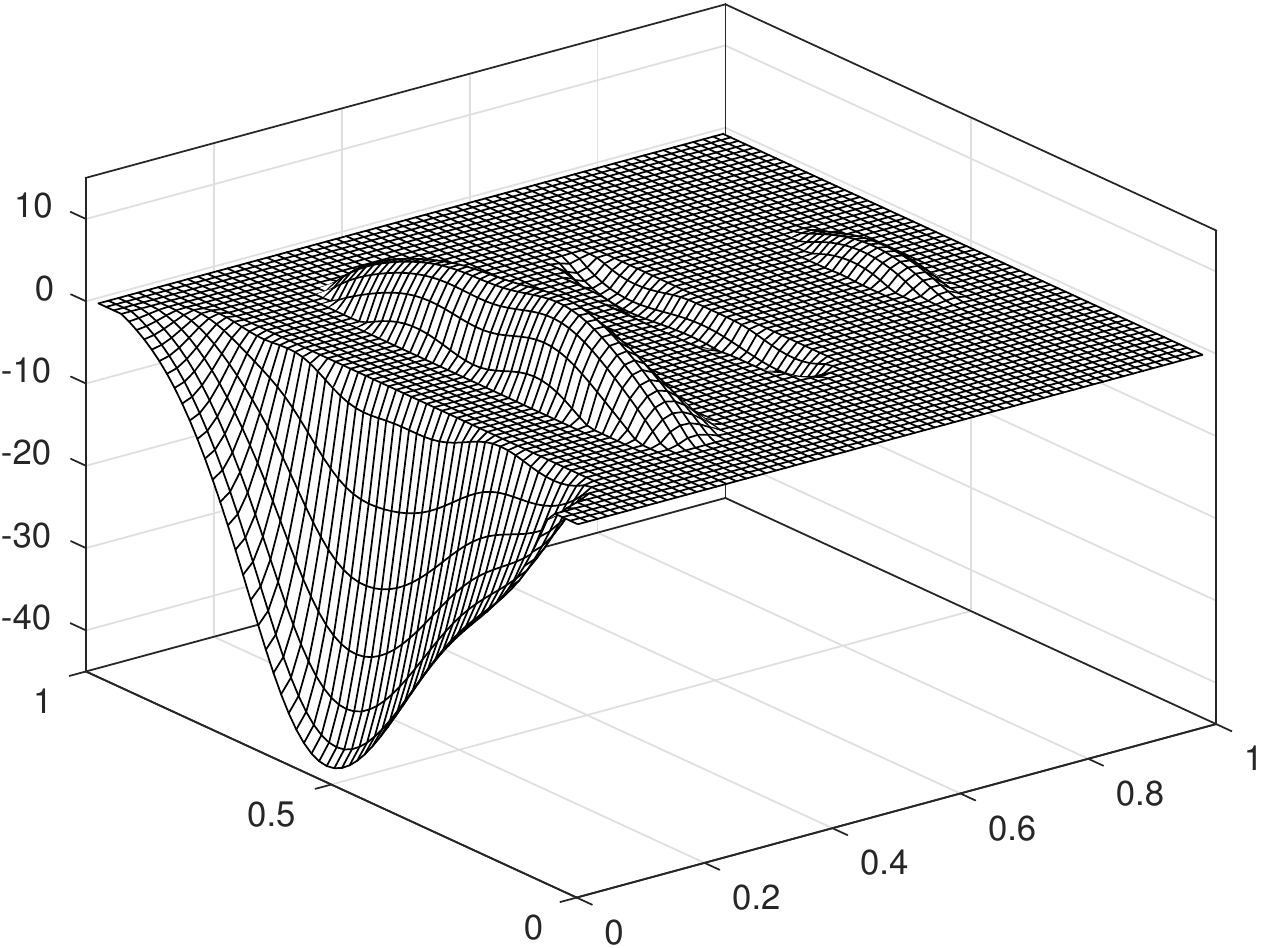}};
\node at (0,-3.5cm)
{\includegraphics[width=0.45\columnwidth]{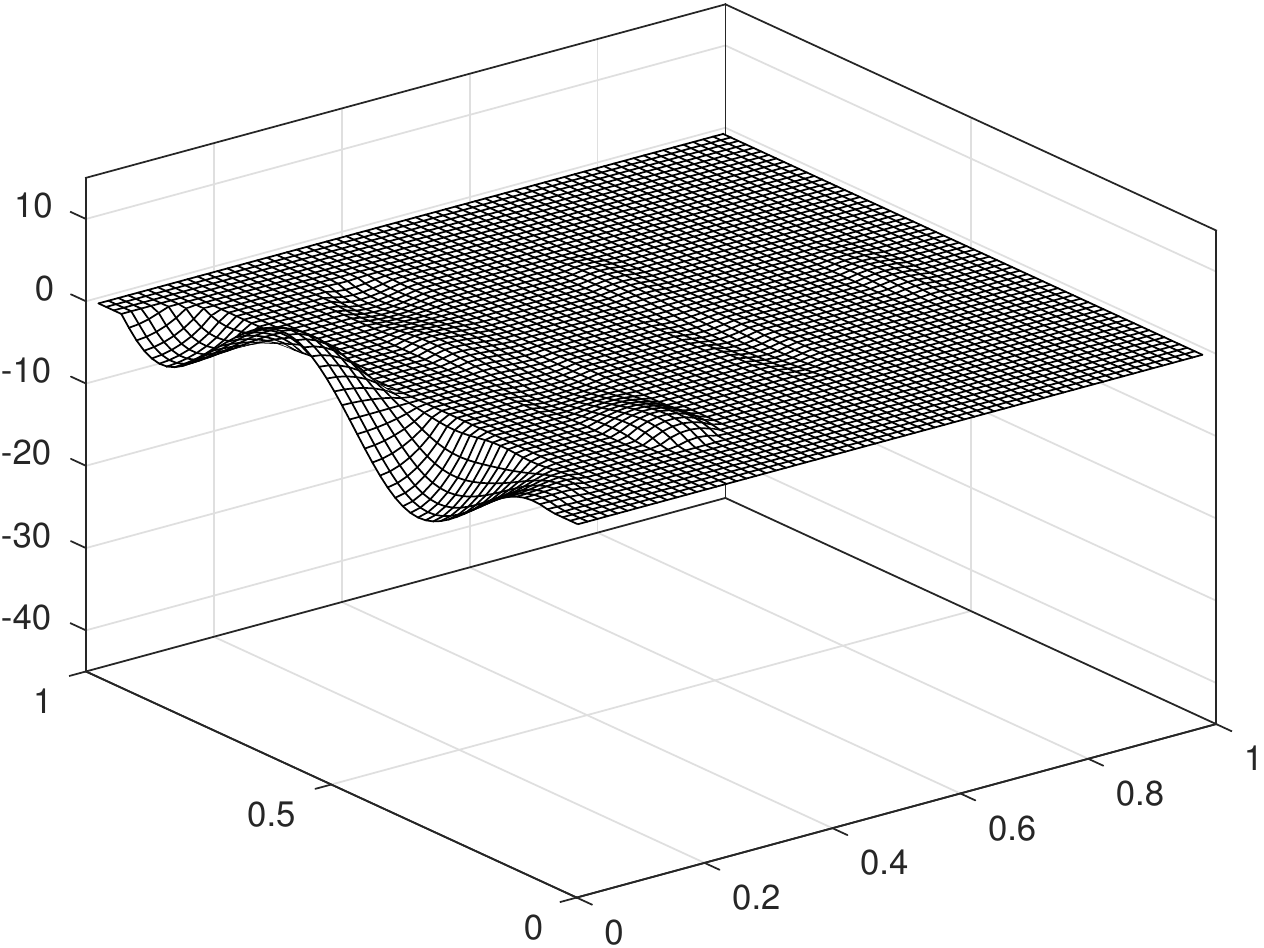}};
\node at (6.5,-3.5cm) {
\includegraphics[width=0.45\columnwidth]{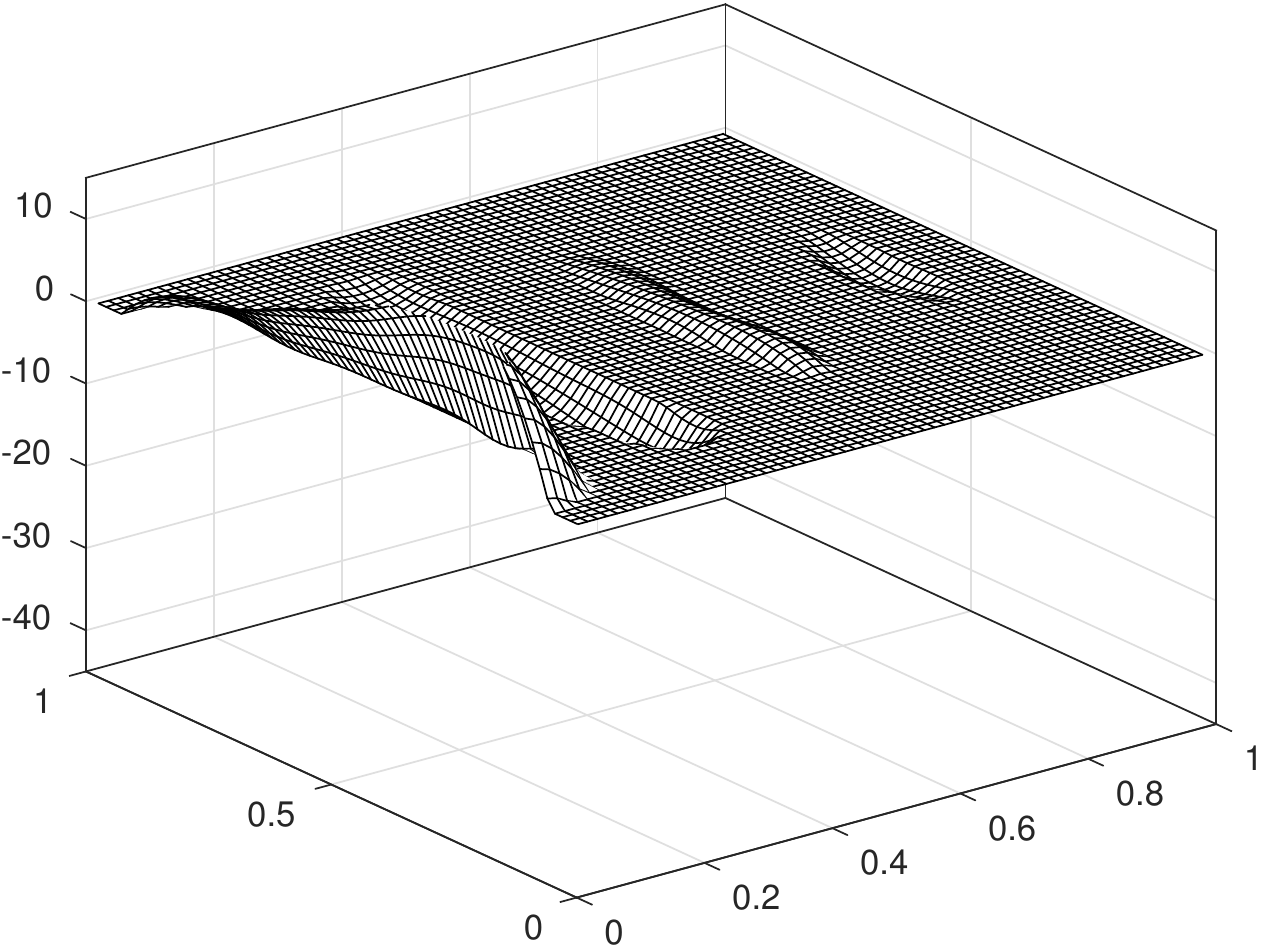}};
\node[color=black] at (-1.5,4cm) {a)};
\draw [color=blue!30!green,fill=white!90!black] (4.5,4cm) rectangle (8.5,4.5)
node[color=blue!30!green,pos=0.5] {b) optimal control for $m_1$};
\draw [color=blue!60!green, fill=white!90!black] (-2,-1.5cm) rectangle (2,-1cm)
node[color=blue!60!green,pos=0.5] {c) optimal control for $m_2$};
\draw [color=blue!90!green, fill=white!90!black] (4.5,-1.5cm) rectangle (8.5,-1cm)
node[color=blue!90!green,pos=0.5] {d) optimal control for $m_3$};
\end{tikzpicture}
\caption{Results for \cref{ex2}: Shown in (a) is the pointwise
  standard deviation $\nu^{-1}$ of the optimal controls. All
  stochastic controls have their support in the gray regions, and
  vanish in the white regions. The figures (b), (c) and (d) show the
  optimal controls corresponding to the same samples of the uncertain
  Neumann boundary condition highlighted in \cref{fig:ex1} (a).  Note
  that the optimal controls are different but have the same sparsity
  structure.\label{fig:ex2}}
\end{figure}
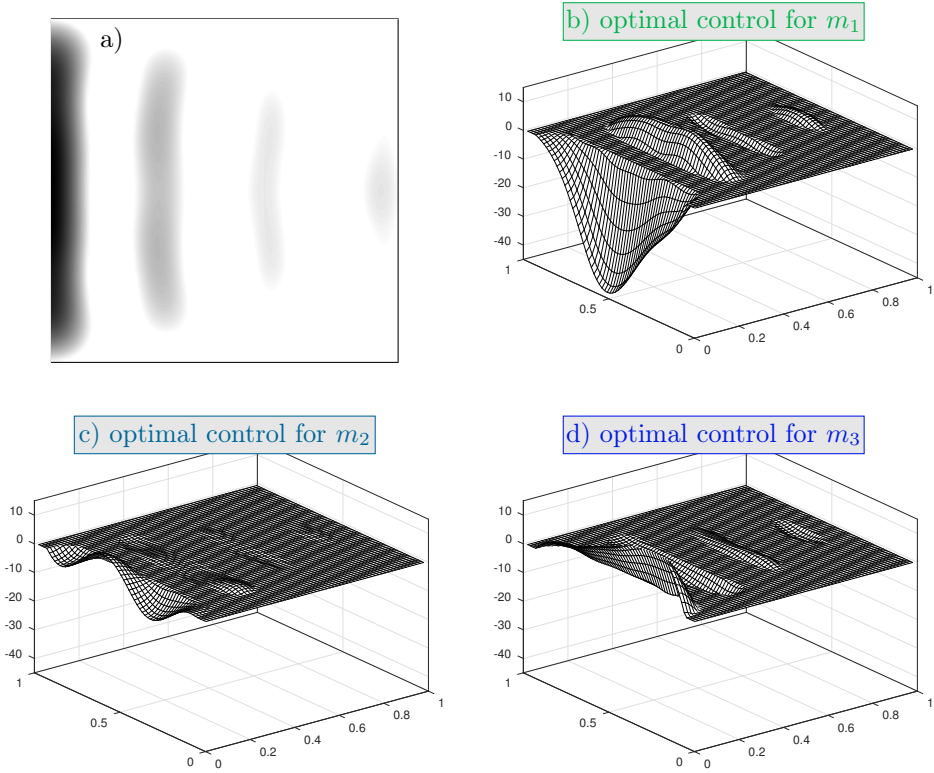

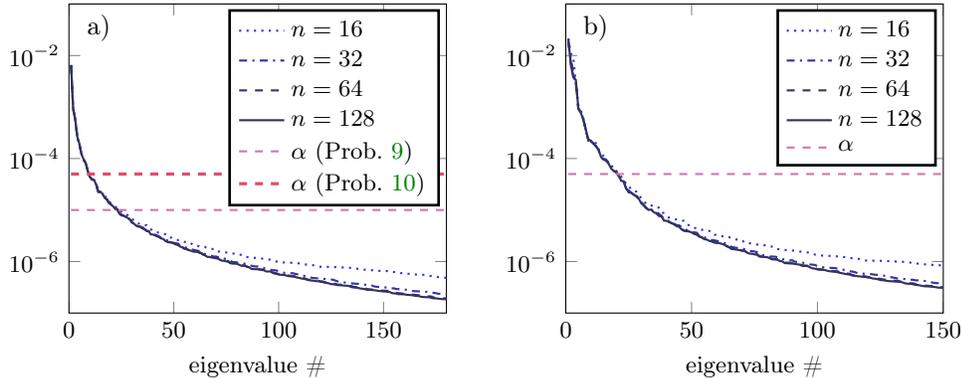
\begin{figure}[tb]\centering
  \begin{tikzpicture}
    \begin{semilogyaxis}[width=6.6cm, xmin=0, xmax=180, ymin=1e-7,
        ymax=0.1, compat=1.3, xlabel= eigenvalue \#, 
        legend style={nodes=right, line width=1pt}, font=\small,
        cells={line width=0.2pt}]
      \addplot[color=black!20!blue!80!white,mark=none,thick,dotted] table [x = x, y =
        eigenvalue]{data_paper/Fig_3/laplace/n_16.txt}; \addlegendentry{$n = 16$}
      \addplot[color=black!40!blue!80!white,mark=none,thick,dashdotted] table [x = x, y =
        eigenvalue]{data_paper/Fig_3/laplace/n_32.txt}; \addlegendentry{$n = 32$}
      \addplot[color=black!60!blue!80!white,mark=none,thick,dashed] table [x = x, y =
        eigenvalue]{data_paper/Fig_3/laplace/n_64.txt}; \addlegendentry{$n = 64$}
      \addplot[color=black!80!blue!80!white,mark=none,thick] table [x = x, y =
        eigenvalue]{data_paper/Fig_3/laplace/n_128.txt}; \addlegendentry{$n =
        128$} 
        \addplot[color=purplecolor!80!red,mark=none,thick,dashed]
      table [x = x, y expr=\thisrow{eigenvalue}*0.2]{data_paper/Fig_3/laplace/alpha.txt};
      \addlegendentry{$\alpha$  (Prob.~\ref{ex1})}
      \addplot[color=purplecolor!50!red,mark=none,very thick,dashed]
      table [x = x, y expr=\thisrow{eigenvalue}]{data_paper/Fig_3/laplace/alpha.txt};
      \addlegendentry{$\alpha$  (Prob.~\ref{ex1.5})}
    \end{semilogyaxis}
    \node[color=black] at (0.4,3.8cm) {a)};
  \end{tikzpicture}
  \hfill
    \begin{tikzpicture}
    \begin{semilogyaxis}[width=6.6cm, xmin=0, xmax=150, ymin=1e-7,
        ymax=0.1, compat=1.3, xlabel= eigenvalue \#, 
        legend style={nodes=right, line width=1pt}, font=\small,
        cells={line width=0.2pt}]
      \addplot[color=black!20!blue!80!white,mark=none,thick,dotted] table [x = x, y =
        eigenvalue]{data_paper/Fig_3/helmholtz/n_16.txt};
      \addlegendentry{$n = 16$}
      \addplot[color=black!40!blue!80!white,mark=none,thick,dashdotted] table [x = x, y =
        eigenvalue]{data_paper/Fig_3/helmholtz/n_32.txt};
      \addlegendentry{$n = 32$}
      \addplot[color=black!60!blue!80!white,mark=none,thick,dashed] table [x = x, y =
        eigenvalue]{data_paper/Fig_3/helmholtz/n_64.txt};
      \addlegendentry{$n = 64$}
      \addplot[color=black!80!blue!80!white,mark=none,thick] table [x = x, y =
        eigenvalue]{data_paper/Fig_3/helmholtz/n_128.txt};
      \addlegendentry{$n = 128$}
        \addplot[color=purplecolor!80!red,mark=none,thick,dashed]
      table [x = x, y = eigenvalue]{data_paper/Fig_3/helmholtz/alpha.txt};
      \addlegendentry{$\alpha$}
    \end{semilogyaxis}
    \node[color=black] at (0.4,3.8cm) {b)};
  \end{tikzpicture}
\caption{Spectra of $A^{-\star}A^{-1}$ for discretizations with
  $n\times n$ points for \cref{ex1}, \cref{ex1.5} (a) and \cref{ex2}
  (b). As reference, we also show the value of $\alpha$, which
  allows estimation of the truncation error as discussed in
  \cref{subsec:truncation}. \label{fig:truncation}}
\end{figure}

\subsection{Qualitative solution properties}\label{sec:qualitative}

Let us first discuss the results of \cref{ex1} shown in
\cref{fig:ex1}. As in the version of this problem not involving
uncertain parameters \cite{Stadler09}, the distributed controls vanish
on parts of the domain. We also find that increasing $\beta$ increases
sparsity (not shown here). The deterministic optimal control, i.e.,
the solution to \eqref{eq:optcon_linear_d1}, looks similar to the
stochastic optimal controls from \cref{fig:ex1} but vanishes near the
Neumann boundary. As expected for the stochastic control problem,
the optimal controls corresponding to different realizations of the
uncertain Neumann data differ, but they share the same sparsity
structure.  Note that the differences between the optimal controls
occur primarily close to the boundary $\partial\D_2$, which is where
the uncertain Neumann data enter in the problem. This local effect of
different Neumann data is due to the locality properties of the
Laplace operator. A different behavior is found for \cref{ex2}. Here,
as can be seen from \cref{fig:ex2}, the optimal controls differ
substantially even far away from $\partial \D_2$, which is a
consequence of the non-local behavior of Helmholtz equation
solutions. In fact, in this problem, the only cause for the controls
to be nonzero is the Neumann boundary data, which impacts the optimal
controls in a non-local manner. We also point out that the solution of
the deterministic version \eqref{eq:optcon_linear_d1} of
\cref{ex2} is zero.

\subsection{Performance of algorithms}\label{sec:performance}
Next, we focus on the accuracy of the low-rank approximations we
employ and the quantitative behavior of our solution algorithms. While
first we compare the low-rank approximations for both, \cref{ex1,ex2},
the remainder of the results shown in this section are for \cref{ex2}. We have verified that the
behavior of the algorithms for \cref{ex1} is similar.

\paragraph{Truncation and low-rank approximation}
Let us first discuss the low-rank approximation of $A^{-\star}A^{-1}$
and the choice of the truncation. In \cref{fig:truncation}, the
spectra of $A^{-\star}A^{-1}$ for \cref{ex1,ex2} are shown for
different mesh resolutions. First, it can be seen that the eigenvalues
converge as the mesh is refined, which is a consequence of the
smoothness of the eigenvectors of the inverse Laplace and Helmholtz
operators. Second, these plots help determine a reasonable truncation
for the low-rank approximation. For that purpose, we recall that from the
error term in \eqref{eq:truncation} it follows that the low-rank error
depends on how small the truncated eigenvalues are compared to the
value of $\alpha$ (also shown in \cref{fig:truncation}.  For our
numerical results, we use low-rank approximations with $r=180$ for
\cref{ex1} and \cref{ex1.5}, $r=150$ for \cref{ex2}. For each example,
we show the exact relative truncation error at the solution and the
upper bound \eqref{eq:truncation2}, in which we truncate the infinite
sums after another $500$ eigenvalues.  For \cref{ex1}, the relative
truncation error rate is $3.3\times 10^{-6}$ while the estimation is
$3.5\times 10^{-2}$. In \cref{ex2}, we have a relative truncation
error of $1.5\times 10^{-4}$ and the theoretical bound is $3.0\times
10^{-2}$.  For the problems with the uncertain boundary data, $\tilde
r=16$ results in approximation of \eqref{eq:E} up to
machine precision. For \cref{ex1.5}, $\tilde r =64$ is used since the
two-dimensional Gaussian random field requires more approximation
vectors. This $\tilde r$ captures $>99.9\%$ of \eqref{eq:E}.
We have numerically verified that the convergence behavior does not
change substantially if higher-rank approximations are used. Moreover,
the optimal controls are visually identical when more basis
functions are used, showing that the error due to truncation is small.

\paragraph{IRLS and over-relaxed IRLS}
In \cref{fig:performance}a, we show the performance of
\cref{alg:reweighting}. We attempt to speed up the algorithm by means of
over-relaxation, using the reweighting function
\begin{equation}\label{eq:overrelax}
\nu^{k+1}:=(1-\theta)\nu^k + \theta \bar\nu^{k+1},
\end{equation}
where $\theta\ge 1$ and $\bar\nu^{k+1}$ is the weight function
computed from the (original) IRLS algorithm. For $\theta=1$, we
recover the original method. Empirical experiments have led us to
choose $\theta=1.5$, which leads to moderately faster convergence, as
can be seen in \cref{fig:performance}a. The IRLS algorithms
converge rather slowly but monotonously, as predicted by the
theory and also observed in other contexts
\cite{ItoKunisch13, DaubechiesDeVoreFornasierEtAl10}.  We find
the convergence behavior of the IRLS algorithm to be largely independent
of $\eps$ and of the discretization mesh size
$N$.

\paragraph{NIRLS algorithm}
Next, we study the performance of the preconditioned Newton-CG
algorithm (\cref{alg:newton}). Since the Newton method is not
guaranteed to converge monotonously and the IRLS algorithm converges
rapidly in early iterations, we first perform 15 over-relaxed IRLS
steps, and then switch to NIRLS. \cref{fig:performance}b--d shows
performance results of the method for different numbers of
preconditioned CG iterations per Newton step, various values of $\eps$
and different discretizations. In \cref{fig:performance}b,c, we show
the norm of the gradient versus the computational cost as discussed in
\cref{subsec:compcost}. We choose one step of the IRLS algorithm
(i.e., one computation of the gradient $\mathcal G_r$) as the unit of
cost. The computational complexity of NIRLS iterations is converted to
this cost unit to allow for a fair comparison between the different
methods. For instance, following the complexity estimates from
\cref{subsec:compcost} and using the specific choices $r=150$ and
$\tilde r=16$ used in this problem, the computational complexity of a
step of NIRLS with $3$ CG iterations is $2.35\times$ the cost of one IRLS 
iteration. This ratio increases to $3.23$ if $8$ CG iterations are used in
each NIRLS step.  As can be seen in \cref{fig:performance}b, the NIRLS
method converges significantly faster than the IRLS method. Moreover,
a small number of CG steps per Newton iteration results in the fastest
convergence in terms of computational cost. Thus, we use 3 CG steps
per Newton iteration for the remaining tests. \cref{fig:performance}c
shows that we observe fast local convergence for every value of $\eps$
and that we observe a mild dependence of the convergence on the value
of $\eps$. Finally, \cref{fig:performance}d compares the convergence
for different mesh discretizations and we observe mesh-independent
convergence behavior, which illustrates the efficiency of the diagonal
preconditioner.

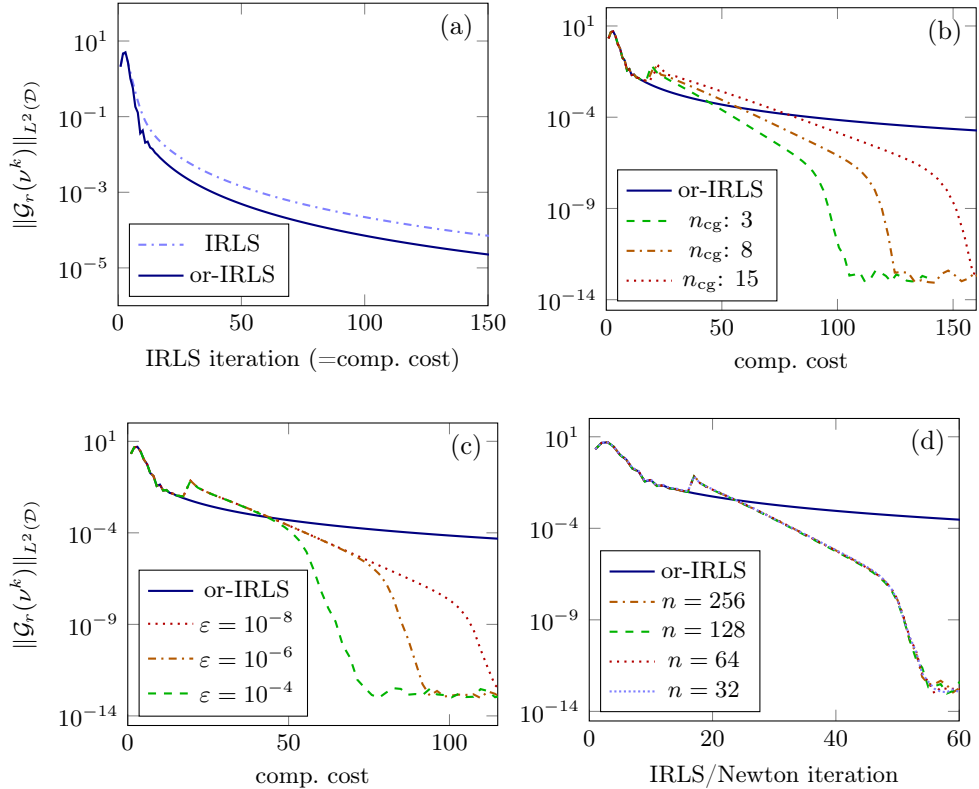
\begin{figure}[tbh]\centering
  \begin{tikzpicture}
    \begin{semilogyaxis}[width=6.5cm, xmin=0, xmax=150, ymin=0.000001,
        ymax=100, compat=1.3, xlabel={IRLS iteration (=comp.\ cost)},
        ylabel = $\|\mathcal G_r(\nu^k)\|_{L^2(\D)}$,
        legend pos=south west, font=\small, cells={line width=0.2pt}]
     \addplot[color=\lightbluecolor,mark=none,thick,dashdotted] table [x = x, y = norm_grad]{data_paper/Fig_4/plot_1/reweighting.txt};
\addlegendentry{IRLS}
\addplot[color=\darkbluecolor,mark=none,thick] table [x = x, y = norm_grad]{data_paper/Fig_4/plot_1/overrelax_1-5.txt};
\addlegendentry{or-IRLS}
    \end{semilogyaxis}
    \node[color=black] at (4.5,3.7cm) {(a)};
  \end{tikzpicture}
  \hfill
    \begin{tikzpicture}
    \begin{semilogyaxis}[width=6.5cm, xmin=0, xmax=160, 
        ymax=100, compat=1.3, xlabel={comp.\ cost}, 
        legend pos=south west, font=\small,
        cells={line width=0.2pt}]
      \addplot[color=\darkbluecolor,mark=none,thick,solid] table [x = x, y = norm_grad]{data_paper/Fig_4/plot_1/overrelax_1-5.txt};
    \addlegendentry{or-IRLS}
    \addplot[color=\greencolor,mark=none,thick,dashed] table [x = x, y = norm_grad]{data_paper/Fig_4/plot_2/cg_3.txt};
    \addlegendentry{$n_{\text{cg}}$: $3$}
    \addplot[color=\orangecolor,mark=none,thick,dashdotted] table [x = x, y = norm_grad]{data_paper/Fig_4/plot_2/cg_8.txt};
    \addlegendentry{$n_{\text{cg}}$: $8$}
    \addplot[color=\redcolor,mark=none,thick,dotted] table [x = x, y = norm_grad]{data_paper/Fig_4/plot_2/cg_15.txt};
    \addlegendentry{$n_{\text{cg}}$: $15$}
    \end{semilogyaxis}
    \node[color=black] at (4.5,3.7cm) {(b)};
  \end{tikzpicture}\\[3ex]
  \begin{tikzpicture}
    \begin{semilogyaxis}[width=6.5cm, xmin=0, xmax=115, 
        ymax=100, compat=1.3, xlabel= {comp.\ cost}, ylabel = $\|\mathcal G_r(\nu^k)\|_{L^2(\D)}$, 
        legend pos=south west, font=\small, cells={line width=0.2pt}]
     \addplot[color=\darkbluecolor,mark=none,thick,solid] table [x = x, y = norm_grad]{data_paper/Fig_4/plot_1/overrelax_1-5.txt};
    \addlegendentry{or-IRLS}
    \addplot[color=\redcolor,mark=none,thick,dotted] table [x = x, y = norm_grad]{data_paper/Fig_4/plot_3/eps_8.txt};
    \addlegendentry{$\eps = 10^{-8}$}
    \addplot[color=\orangecolor,mark=none,thick,dashdotted] table [x = x, y = norm_grad]{data_paper/Fig_4/plot_3/eps_6.txt};
    \addlegendentry{$\eps = 10^{-6}$}
    \addplot[color=\greencolor,mark=none,thick,dashed] table [x = x, y = norm_grad]{data_paper/Fig_4/plot_3/eps_4.txt};
    \addlegendentry{$\eps = 10^{-4}$}   
   \end{semilogyaxis}
   \node[color=black] at (4.5,3.7cm) {(c)};
  \end{tikzpicture}
  \hfill
    \begin{tikzpicture}
    \begin{semilogyaxis}[width=6.5cm, xmin=0, xmax=60, 
        ymax=100, compat=1.3, xlabel={IRLS/Newton iteration}, 
        legend pos=south west, font=\small,
        cells={line width=0.2pt}]
        \addplot[color=\darkbluecolor,mark=none,thick,solid] table [x = x, y = norm_grad]{data_paper/Fig_4/plot_1/overrelax_1-5.txt};
    \addlegendentry{or-IRLS}
      \addplot[color=\orangecolor,mark=none,thick,dashdotted] table [x = x, y = norm_grad]{data_paper/Fig_4/plot_4/n_256.txt};
    \addlegendentry{$n = 256$}
    \addplot[color=\greencolor,mark=none,thick,dashed] table [x = x, y = norm_grad]{data_paper/Fig_4/plot_4/n_128.txt};
    \addlegendentry{$n = 128$}
    \addplot[color=\redcolor,mark=none,thick,dotted] table [x = x, y = norm_grad]{data_paper/Fig_4/plot_4/n_64.txt};
    \addlegendentry{$n = 64$}
    \addplot[color=\lightbluecolor,mark=none,thick,densely dotted] table [x = x, y = norm_grad]{data_paper/Fig_4/plot_4/n_32.txt};
    \addlegendentry{$n = 32$}    
    \end{semilogyaxis}
    \node[color=black] at (4.5,3.7cm) {(d)};
  \end{tikzpicture}
    \caption{Convergence behavior of algorithms for \cref{ex2}.
      Shown in (a) is the reduction of the norm of the gradient for
      IRLS and over-relaxed IRLS (or-IRLS) with $\theta=1.5$ (see
      \eqref{eq:overrelax}). Shown in (b) is a comparison of the
      performance of or-IRLS and the
      preconditioned Newton-CG method NIRLS for different numbers of CG
      iterations per Newton step, where we fix $n=128$,
      $\eps=10^{-7}$.  The figure in (c) compares the convergence of
      the preconditioned Newton-CG method with different $\eps$ for
      $n=128$ and 3 CG iterations per Newton step. Shown in (d) is a
      comparison of the convergence of or-IRLS and the NIRLS method for
      different mesh sizes $n$, $\eps=10^{-7}$ and 3 CG iterations per
      Newton step. As discussed in \cref{subsec:compcost}, in (a)-(c), we
      use the computational work required for one IRLS iteration as
      unit for the $x$-axis to compare the computational complexity of
      the IRLS algorithms and its Newton variants. In (d), we use
      the iteration number as unit for the $x$-axis since we study how the number
      of iterations changes for different discretizations.
      \label{fig:performance}}
\end{figure}

\section{Discussion and remarks}
First, let us discuss the role of $\alpha>0$. This parameter plays a
significant role in our problem formulation, the proposed solution
algorithms and their analysis. Positivity of $\alpha$ is required for the
deterministic problem to be formulated in an $L^2$-Hilbert space
framework rather than over a
space of measures \cite{Casas17,Pieper15}. Additionally, $\alpha>0$ 
plays a crucial role for the truncation of the spectral expansion of
$A^{-\star}A^{-1}$ since we show that the
truncation error is small when the truncated eigenvalues are small
compared to $\alpha$. Both aspects are related to the regularizing effect
positive values of $\alpha$ have on the controls.

Second, it would be desirable to include control bounds in the
proposed algorithms.  However, it is
not obvious how to achieve this without resorting to more general
algorithms that also apply to nonlinear problems and use random
space approximations, such as stochastic Galerkin/collocation or Monte
Carlo methods. The main difference between bound constraints and the
shared sparsity term is that bound constraints apply to the
controls individually for each random event, while the
sparsity term involves integration over the probability
space.

Third, the proposed approach can be generalized to uncertain
parameters that do not follow a Gaussian distribution. As long as for
the resulting distribution of the controls, $\|u\|_\Omega(\bs x)$ can
be computed efficiently, the reweighting algorithms can be applied to
compute jointly sparse controls.

We believe that several questions raised in this paper
deserve further research. For instance, while challenging,
extension to nonlinear problems are worthwhile pursuing, as well as
the question whether the shared sparsity requirement can help to
reduce the effective dimension of non-linear problems. Other
interesting questions include a study of the spatial discretization of
the problem, extensions to parabolic governing equations possibly
combined with directional sparsity, problem formulation
and algorithms for $\alpha=0$, and the question whether
regularization with $\eps>0$ can be avoided.  A question of
potentially general interest is
if the Newton variant of the IRLS algorithm can accelerate the
solution of other non-smooth optimization problems for which currently
reweighting algorithms are used.

\section*{Acknowledgments}
The authors would like to thank two anonymous referees for their
thoughtful comments and suggestions.  GS would like to
additionally thank Noemi Petra, Alen Alexanderian and Kazufumi Ito for
valuable discussions.

\bibliographystyle{siamplain}
\bibliography{ccgo,not_in_ccgo}

\end{document}